\documentclass[11pt]{article}
\usepackage{amsmath,amssymb,amsthm} 
\usepackage{bbm}
\usepackage{stmaryrd}
\usepackage{paralist}
\usepackage{hyperref}
\usepackage{graphicx}
\usepackage{color}
\setdefaultenum{(i)}{}{}{}
\usepackage[margin=1in]{geometry} 
\usepackage{stmaryrd}
\usepackage{algorithm}
\usepackage{algorithmic}
\usepackage{paralist}
\usepackage{hyperref}
\usepackage{color}
\usepackage[toc,page]{appendix}
\usepackage{pgf}
\usepackage{tikz}
\usetikzlibrary{arrows,calc,automata,positioning,topaths}
\usepackage[nottoc, notlof, notlot]{tocbibind}
\usepackage{color}
\usepackage{todonotes}

\setdefaultenum{(i)}{}{}{}
\usepackage[margin=1in]{geometry}
\numberwithin{equation}{section}

\renewenvironment{thebibliography}[1]{%
\begin{oldthebibliography}{#1}%
\setlength{\baselineskip}{.8em}
\linespread{.8}
\small
\setlength{\parskip}{0ex}%
\setlength{\itemsep}{.1em}%
}%
{%
\end{oldthebibliography}%
}

\newtheorem{definition}{Definition}[section]
\newtheorem{proposition}[definition]{Proposition}
\newtheorem{theorem}[definition]{Theorem}

\newtheorem{lemma}[definition]{Lemma}

\newtheorem{assumption}{Assumption}

\theoremstyle{definition}
\newtheorem{remark}[definition]{Remark}

\newcommand{\bE}{\mathbb{E}}
\newcommand{\bF}{\mathbb{F}}
\newcommand{\bG}{\mathbb{G}}

\newcommand{\bN}{\mathbb{N}}
\newcommand{\bP}{\mathbb{P}}
\newcommand{\bQ}{\mathbb{Q}}
\newcommand{\bR}{\mathbb{R}}

\newcommand{\cF}{\mathcal{F}}
\newcommand{\cG}{\mathcal{G}}

\newcommand{\cS}{\mathcal{S}}

\newcommand{\epsi}{\varepsilon}
\newcommand{\one}{\mathbbm{1}}
\newcommand{\comm}[1]{}

\DeclareMathOperator*{\plim}{\bP-lim}
\DeclareMathOperator*{\sign}{\mathrm{sgn}}
\begin{document}

\title{Scaling Limits of Processes with Fast Nonlinear Mean Reversion\footnote{We are grateful to two anonymous referees for extremely detailed and helpful comments that have greatly improved the final version of this manuscript.}}

\author{
Thomas Cay\'e\thanks{Dublin City University, School of Mathematical Sciences, Glasnevin, Dublin 9, Ireland, email \texttt{thomas.caye@dcu.ie}. Partly supported by the Swiss National Science Foundation (SNF) under grant 175133.}
\and
Martin Herdegen\thanks{University of Warwick, Department of Statistics, Coventry, CV4 7AL, UK, email
 \texttt{M.Herdegen@warwick.ac.uk}. Partly supported by the Swiss National Science Foundation (SNF) under grant 150101.}
\and
Johannes Muhle-Karbe\thanks{Imperial College London, Department of Mathematics, London, SW7 1NE, UK, email \texttt{j.muhle-karbe@imperial.ac.uk}. Partially supported by the CFM-Imperial Institute of Quantitative Finance. Parts of this paper were written while this author was visiting ETH Z\"urich; he is grateful to H.M.~Soner and the Forschungsinstitut f\"ur Mathematik for their hospitality.}
}

\maketitle

\begin{abstract}
We derive scaling limits for integral functionals of It\^o processes with fast nonlinear mean-reversion speeds. In these limits, the fast mean-reverting process is ``averaged out'' by integrating against its invariant measure. The convergence is uniformly in probability and, under mild integrability conditions, also in $\cS^p$. These results are a crucial building block for the analysis of portfolio choice models with small superlinear transaction costs, carried out in the companion paper of the present study \cite{caye.al.17}.
\end{abstract}



\section{Introduction and Main Results}

\paragraph{Motivation}

Superlinear trading costs play an important role in financial engineering as reduced-form models for the adverse price impact generated by large trades, cf., e.g., \cite{almgren.chriss.01,almgren.03} as well as many more recent studies. In this context, optimal policies typically prescribe to track some ``target portfolio'' at a finite, absolutely-continuous rate \cite{garleanu.pedersen.13,garleanu.pedersen.16,bank.soner.voss.16,cai.al.15,cai.al.16,almgren.li.16,moreau.al.15,guasoni.weber.15a,guasoni.weber.15b}. If trading costs are quadratic, this trading speed is linear in the deviation from the target \cite{garleanu.pedersen.13,garleanu.pedersen.16,moreau.al.15,guasoni.weber.15a}, leading to Ornstein-Uhlenbeck dynamics of the deviation in the small-cost limit. The well-known properties of this process can in turn be used to analyze the asymptotic performance of the corresponding tracking portfolios \cite{cai.al.15,cai.al.16}.

However, empirical studies suggest that actual trading costs are superlinear but also subquadratic, corresponding to the ``square-root law'' for price impact advocated by many practitioners~\cite{lillo.al.03,almgren.al.05}. For such trading costs, optimal trading rates become nonlinear: compared to the quadratic case, trading slows down near the target, where costs are higher. Conversely, trading is sped up far away from the target, where costs are comparatively lower. In the limiting case where the trading costs become proportional, this leads to ``reflecting'' singular controls: no transactions at all inside some ``no-trade region'' and instantaneous reflection by trading at an ``infinite rate'' once its boundaries are breached \cite{davis.norman.90}. The controlled deviation from the target process in turn follows a reflected diffusion process. In the small-cost limit, its study boils down to the analysis of doubly-reflected Brownian motion. Whence, the asymptotic analysis of the corresponding singular tracking strategies can also be performed by appealing to well-known probabilitic results~ \cite{janecek.shreve.10,ahrens.15,cai.al.15,cai.al.16}. In addition to these law of large numbers type results, even a central limit theorem has been derived for the asymptotic performance in this context~\cite{cai.fukasawa.16}. 

However, such results are not available for the empirically most relevant transaction costs that fall between linear and quadratic. The corresponding controlled deviations then correspond to processes with \emph{nonlinear} mean-reversions speeds \cite{guasoni.weber.15b,caye.al.17}. In this paper, we determine the scaling limits of such processes in the regime where the mean-reversion speed becomes large. This is a main building block for the derivation of asymptotically optimal trading strategies with small superlinear trading costs in the companion paper of the present study \cite{caye.al.17}. There, the results from the present paper are used to compute both the expected transaction costs incurred by a given tracking strategy, and its average squared displacement from the target. Trading off these two terms in an optimal manner in turn leads to the asymptotically optimal performance.

Our scaling limits also contribute to the classical literature on ``averaging results'', where a ``fast variable''  is averaged out appropriately as it oscillates faster and faster. Results of this kind were first developed by \cite{stratonovich.63,stratonovich.67,khasminskii.66,papanicolaou.al.77}; textbook treatments can be found in \cite{ethier.kurtz.86,skorokhod.89}. Our probabilistic approach allows us to extend existing results of, e.g,~\cite{pardoux.veretennikov.01, pardoux.veretennikov.03} that require a Markovian or semimartingale structure on the drift and diffusion coefficients of the SDEs to general unbounded drift and diffusion coefficients.

\paragraph{Setting} Let $(\Omega,\cF, (\cF_t)_{t \in [0,T]}, \bP)$ be a filtered probability space satisfying the usual conditions. For a (small) parameter $\epsi > 0$, we consider the following family of mean-reverting stochastic differential equations (SDEs):
\begin{equation}\label{eq:SDE}
dX^\epsi_t= \left(b_t-\frac{c_tL_t}{\epsi} g\left(\frac{M_t X^{\epsi}_t}{\epsi}\right)\right) dt + \sqrt{c_t} dW_t, \quad X^\epsi_0=x^\epsi_0, \quad t\in [0,T].
\end{equation}
Here, $W$ is a standard Brownian motion, the processes $b, c, L, M$ are adapted and continuous and $c, L, M$ are positive. 
The function $g$ describes the nonlinear nature of the mean reversion. It is locally Lipschitz, odd, non-decreasing and nonnegative on $\bR_+$ (so that $X^\epsi$ is indeed always steered back towards zero), and of superlinear polynomial growth at infinity:\footnote{In sufficiently regular Markovian settings, weaker growth conditions suffice to ensure the recurrence of the diffusion $X^\epsi$, cf.~\cite{pardoux.veretennikov.01}). Weaker versions of our results under such conditions are discussed in Remark~\ref{rem:corollary.Vere}.}
\begin{equation}
\label{eq:growth.g}
\liminf_{x \to \infty} \frac{g(x)}{x^q} >0, \quad \mbox{for some $q\geqslant 1$}.
\end{equation}
(Larger values of $q$ correspond to faster mean-reversion for large deviations, which allow to weaken the integrability requirements that need to be imposed on the other primitives of the model in Assumption \ref{ass.integrability}.) The antiderivative of $g$ is denoted by $G(x)=\int_{0}^{x}g(y)dy$. The processes $b, c, L, M$ and the function $g$ are all independent of the scaling parameter $\epsi$.  In contrast, the initial value $x^\epsi_0 \in \bR$ of $X^\epsi$ may depend on $\epsi$ as long as $\limsup_{\epsi \to 0} |x^\epsi_0|/\epsi < \infty$.

\paragraph{Interpretation} In the context of portfolio optimization with small nonlinear trading costs~\cite{caye.al.17}, $X^\varepsilon$ corresponds to the deviation of the frictional portfolio process from its frictionless counterpart. $b$ and $c$ correspond to the drift and diffusion coefficients of this target position, whereas the mean-reverting part of the drift of \eqref{eq:SDE} is the absolutely continuous control applied to steer the actual position in its direction. Up to rescaling, the asymptotic parameter $\epsi$ corresponds to the size of the trading cost. As it decreases, the mean-reversion becomes faster and faster and the frictional positions eventually converge to their frictionless counterparts. 

In order to determine asymptotically optimal portfolios, the average squared values of the deviations \eqref{eq:SDE} need to be traded off against the corresponding trading costs (a nonlinear functional of the control that is applied). In the present paper, we develop limit theorems that allow, in particular, to compute both of these terms in closed form at the leading-order for small $\epsi$.

\paragraph{Results} We first establish that the SDE~\eqref{eq:SDE} is well posed despite the superlinear growth of its drift rate at infinity; for better readability, the proof of this result is deferred to Section~\ref{section.proof.ex}.

\begin{proposition}\label{prop:sde}
	For each $\epsi > 0$, and $x^\epsi_0\in \bR$, there exists a unique strong solution of the SDE~\eqref{eq:SDE}.
\end{proposition}

In order to formulate our scaling limits for the quickly mean-reverting processes~\eqref{eq:SDE}, we fix a function 
$f : \mathbb{R} \to \mathbb{R}$ that is even,\footnote{The assumption that $f$ is even is crucially used in the arguments leading to Lemma~\ref{lem:sandwich:2}. In its proof, we bound the process $(X^\epsi)^2$ from above and below by positive processes, and the fact that $f$ is even allows us in turn to derive bounds for $f(X^\epsi)$.} of finite variation on compacts and satisfies the following polynomial growth condition:
\begin{equation}
\label{eq:condition.f}
\left|f(x)\right|\leqslant C_f(\left|x\right|^{q'}+1), \quad \mbox{for some $q'\geqslant 0$ and some $C_f > 0$.}
\end{equation}

We can now formulate our first scaling limit, whose proof is delegated to Section~\ref{sec:proofs}. Since this first result only asserts convergence in probability, it does not require any integrability assumptions on the primitives of the SDE~\eqref{eq:SDE}.

\begin{theorem}
	\label{theorem:up}
	Let $(H_t)_{t \in [0, T]}$ and $(K_t)_{t \in [0, T]}$ be nonnegative, continuous, adapted processes. Then, as $\epsi \to 0$, the following limit holds uniformly in probability:
	\begin{equation}
	\label{eq:thm:up}
	\int_0^{\cdot} H_s f\bigg(\frac{K_s X^\epsi_s}{\epsi}\bigg)ds \longrightarrow \int_0^\cdot H_s \frac{\int_{\bR} f\big(\frac{K_s}{M_s}y\big) \exp\big(-2 \frac{L_s}{M_s} G(y)\big)dy}{\int_{\bR}\exp\big(-2 \frac{L_s}{M_s} G(y)\big)dy}ds.
	\end{equation}
\end{theorem}

The intuition for the limit~\eqref{eq:thm:up} is the following. As the mean-reversion speed of the processes \eqref{eq:SDE} becomes faster and faster, one can essentially treat the ``slow'' processes $H$, $K$, $b$, $c$, $L$, $M$ as constant on each time step in a fine partition of $[0,T]$. In contrast, the rescaled process $X^\epsi/\epsi$ converges to a mean-reverting one-dimensional diffusion on each of these infinitesimal intervals. The limit~\eqref{eq:thm:up} asserts that, as $\epsi \to 0$, one can replace $X^\epsi/\epsi$ in the integrand by an integral with respect to the stationary distribution of this limiting process, which is given in terms of its normalized and rescaled speed measure.\footnote{In the limiting case where the rescaled process is a reflected diffusion, the stationary law is uniform, as exploited in \cite{ahrens.15,cai.al.15,cai.al.16}. Conversely, if the rescaled process has linear mean reversion, it is Gaussian, compare \cite{cai.al.15,cai.al.16}. More general averaging results for Markovian settings are developed in \cite{pardoux.veretennikov.01,pardoux.veretennikov.03}, for example.}

Results similar to Theorem~\ref{theorem:up} have been developed under abstract assumptions and verified for the simplest case of linear mean-reversion speeds in \cite[Section 3.2]{cai.al.16}. Here, we extend this to the nonlinear trading speeds arising naturally in the context of square-root price impact and show that  convergence in probability remains valid under minimal assumptions also in this case.

Our second main result provides conditions under which this limit theorem can be lifted to convergence in expectation. (For better readability, the proof is delegated to Section~\ref{sec:proofs}.) This is needed to study the small-cost asymptotics of expected utility maximization problems rather than the pathwise, quadratic criteria of \cite{cai.al.15}.  Unlike for singularly- or impulse-controlled deviations (where integrability is inherited directly from the corresponding trading boundaries \cite{ahrens.15,feodoria.16,herdegen.muhlekarbe.17}), this necessitates further delicate estimates that require the following integrability assumptions on the processes $b, c, L, M, H, K$ appearing in the SDE for $X^\epsi$ and the scaling limit~\eqref{eq:thm:up}:\footnote{Assumptions \ref{ass.novikov} and \ref{ass.integrability} are required to apply H\"older's inequality in Lemmas~\ref{lemma.uniform.moment.bar}, \ref{lem:U.I}, \ref{lem.max.inequality}, and \ref{lem:app:integrability estimate}.}

\begin{assumption}\label{ass.novikov} 
$\bE[\exp(8\int_{0}^{T}\frac{b_t^2}{c_t}dt)]<\infty$.
\end{assumption}

\begin{assumption}\label{ass.integrability}
For $p\geqslant 1$, there exists $\eta>0$ such that 
\begin{align}
&\bE\left[\sup_{u\in[0,T]}\left(L_{u}c_{u}\wedge M_{u}\right)^{-\frac{2(q+1)( 4pq'(1+\eta)\vee 2)}{q-1}}\right]+\bE\left[\sup_{u \in [0, T]}c^{8pq'(1+\eta)\vee 4}_u\right]\notag\\
&+\bE\left[\int_{0}^{T}H^{p(1+\eta)}_tdt\right]+\bE\left[\int_{0}^{T}\left(H_tK^{q'}_t\right)^{ 2p(1+\eta)}dt\right]+\bE\left[\int_{0}^{T}\left(\frac{1}{L_s \wedge M_s}\right)^{4 pq^\prime(1+\eta)}dt\right] \notag \\
&+\bE\left[\int_{0}^{T}\left(\frac{M_t}{L_t}\right)^{ 4p(1+\eta)}dt\right]  +\bE\left[\int_{0}^{T}\exp\left(\eta \frac{L_t}{M_t} \right) dt\right] < \infty.\notag
\end{align}
\end{assumption}
\begin{remark}
	\label{rem:integrability}
	If $q=1$ in \eqref{eq:growth.g}, then the first moment condition of Assumption~\ref{ass.integrability} is to be understood as $\text{essinf}_{u\in[0,T]}(L_{u}c_{u}\wedge M_{u}) >0$.
\end{remark}

We now turn to our second main result, which is the main tool for the analysis of asymptotically optimal trading strategies with nonlinear trading costs in the companion paper of the present study~\cite{caye.al.17}:

\begin{theorem}
	\label{theorem}
	Suppose that Assumptions \ref{ass.novikov} and \ref{ass.integrability}  are satisfied for some $p \geqslant 1$. Then the scaling limit~\eqref{eq:thm:up} also holds in $\cS^p([0,T])$.\footnote{$\cS^p\left([0,T]\right)$ denotes the set of c\`adl\`ag adapted processes whose running supremum on $[0,T]$ has finite absolute $p$-th moments. A sequence of processes $(Y^{(n)})_{n\in\bN}$ converges to $Y$ in $\cS^p([0,T])$ if $\lim_{n\to\infty} \bE[\sup_{t\in[0,T]}|Y^{(n)}_t-Y_t|^p]=0$. Assumption~\ref{ass.integrability} guarantees that the limit in Theorem~\ref{theorem} is finite, cf.~Lemma~\ref{lem:app:integrability estimate}.}
\end{theorem}

\begin{remark}\label{rem:corollary.Vere}
	(i) For sufficiently regular Markovian settings, Theorem~\ref{theorem:up} with convergence in probability for all $t\in[0,T]$ and Theorem~\ref{theorem} with convergence in $L^p$ instead of $\cS^p$ are consequences of \cite[Theorem 4]{pardoux.veretennikov.03}. 
	
	(ii)  \cite[Theorem 4]{pardoux.veretennikov.03} requires weaker growth conditions for the function $g$ than~\eqref{eq:growth.g}. Under additional assumptions on the diffusion coefficient of $X^\epsi$, our results can be extended in this direction. (We are grateful to one of the referees for pointing this out.) To illustrate this, suppose that $c_t=\sigma^2>0$ is constant and that, instead of \eqref{eq:growth.g} for $g$, there exists $C,\tilde{C}>0$ such that, for all $x>\tilde{C}$,
	\begin{align}
	\label{eq:condition.Veretennikov}
		L_t	g(M_t x)\geqslant C, \quad t\in [0,T], \quad \mbox{and} \quad G(x)\geqslant Cx.
	\end{align}
	Assume moreover that the integrability conditions from Assumption~\ref{ass.integrability} are satisfied (except for the first two terms, which are not needed in this case). Then, the convergence result in Theorem~\ref{theorem:up} holds in probability for all $t\in[0,T]$ instead of uniformly on compacts in probability. Likewise, the convergence in Theorem~\ref{theorem} holds in $L^p$ instead of in $\cS^p$.
	
	The above extension heavily relies on the comparison theorem for SDEs; cf.~Remark \ref{rem:comparison:Vere}. For more general diffusion coefficients, one would have to extend the arguments in Veretennikov \cite{veretennikov.97} to non-constant volatilities. As the latter will typically depend on the asymptotic parameter $\epsi$, this is rather challenging.
\end{remark}

The remainder of this article is organized as follows. Section~\ref{sec:reduction} describes a localization argument that allows to reduce the analysis to the case of bounded coefficients. This crucially relies on a uniform integrability result (Lemma~\ref{lem:U.I}) established in Appendix~\ref{app:UI}. Section~\ref{section.proof.ex} contains the proof of Proposition~\ref{prop:sde} and introduces the fundamental time and measure changes used throughout the rest of the paper. Section~\ref{sec:proofs} contains the proof of our main result, Theorem \ref{theorem}; due to the localization argument from Section~\ref{sec:reduction}, the proof of Theorem \ref{theorem:up} turns out to be a simple corollary. The proof of Theorem~\ref{theorem} is rather delicate. First, in Section \ref{sec:local estimation} we establish a local scaling limit on a small interval. This combines sandwiching and approximation arguments with ergodic theory for one-dimensional diffusions. A key difficulty for the ergodic result in Lemma \ref{lem:ergodic elementary} is that the diffusions under consideration also depend on the time horizon (via the small parameter $\epsi$). In a second step, we concatenate the local limit theorems to a global limit theorem in Section \ref{section.concatenation}. Appendix \ref{app:UI} contains the already mentioned result on uniform integrability. Appendix \ref{app:maximal inequality} proves a maximal inequality in the spirit of Peskir~\cite{peskir.01} necessary for the concatenation argument in Section~\ref{section.concatenation}. Appendix~\ref{app:comparison} establishes some comparison and existence results for SDEs, and Appendix~\ref{app:auxiliary} concludes with some other auxiliary results.

\section{Reduction to Bounded Coefficients}
\label{sec:reduction}
In this section, we show why -- up to the results on uniform integrability in Appendix \ref{app:UI} -- it suffices to establish all results for the case of bounded coefficients. For fixed $\kappa \in (0, 1)$, define the stopping time $\tau^\kappa$ by
\begin{align}
\label{def:tau kappa}
\tau^\kappa := \inf\bigg\{&t \in [0, T]: \int_0^t \frac{b_u^2}{c_u} du > \frac{1}{\kappa}, H_t > \frac{1}{\kappa}, K_t > \frac{1}{\kappa}, \notag \\
 &c_t \notin \left[\kappa, \frac{1}{\kappa}\right], L_t \notin \left[\kappa, \frac{1}{\kappa}\right], M_t \notin \left[\kappa, \frac{1}{\kappa}\right] \bigg\} \wedge T.
\end{align}
Note that the stopping times $\tau^\kappa$ are non-increasing in $\kappa$, and for every $\omega$ outside of a null set, there is a $\kappa(\omega)>0$ such that $\tau^\kappa=T$ for $0<\kappa\leqslant \kappa(\omega)$, by continuity of $b, c, H, K, L, M$ and positivity of $c, L, M$. This in turn implies that
\begin{equation}
\label{eq:tau kappa}
\lim_{\kappa \to  0}\bP[\tau^\kappa = T] =  1.
\end{equation}
For $\Upsilon \in \{b, c, H, K, L, M \}$, introduce the stopped processes 
\begin{equation*}
\Upsilon^\kappa_t := \Upsilon_t^{\tau^\kappa}, \quad t \in [0, T],
\end{equation*}
and consider for fixed $\epsi > 0$ the corresponding SDE
\begin{equation}\label{eq:SDE:kappa}
dX^{\epsi,\kappa}_t= \left(b^\kappa_t-\frac{c^\kappa_tL^\kappa_t}{\epsi} g\left(\frac{M^\kappa_t X^{\epsi,\kappa}_t}{\epsi}\right)\right) dt + \sqrt{c^\kappa_t} dW_t, \quad X^{\epsi,\kappa}_0=x_0^\epsi, \quad t\in [0,T].
\end{equation}
Note that the SDE for $X^{\epsi,\kappa}$ coincides with the SDE \eqref{eq:SDE} for $X^\epsi$ on $\llbracket 0, \tau^\kappa \rrbracket$.

Now suppose that Proposition \ref{prop:sde} as well as Theorems \ref{theorem:up} and \ref{theorem} have been established for $X^{\epsi, \kappa}$, $H^\kappa$, and $K^\kappa$ for each fixed $\kappa \in (0, 1)$.
Then Proposition \ref{prop:sde} for $X^{\epsi}$ follows from \eqref{eq:tau kappa} and pathwise uniqueness of strong solutions. In particular, we have
\begin{equation}
\label{eq:X epsi kappa}
X^\epsi  = X^{\epsi,\kappa}   \text{ on } \llbracket 0, \tau^\kappa \rrbracket.
\end{equation}

Next, to establish Theorem \ref{theorem:up} for $X^{\epsi}$, set
\begin{align*}
w_t &:= \frac{\int_{\bR} f\left(\frac{K_t}{M_t}y\right) \exp\left(-2 \frac{L_t}{M_t} G(y)\right)dy}{\int_{\bR}\exp\left(-2 \frac{L_t}{M_t} G(y)\right)dy}, \quad w^\kappa_t := \frac{\int_{\bR} f\left(\frac{K^\kappa_t}{M^\kappa_t}y\right) \exp\left(-2 \frac{L^\kappa_t}{M^\kappa_t} G(y)\right)dy}{\int_{\bR}\exp\left(-2 \frac{L^\kappa_t}{M^\kappa_t} G(y)\right)dy},
\end{align*}
for $t \in [0, T]$ and $\kappa \in (0, 1)$. Then \eqref{eq:X epsi kappa} and Theorem \ref{theorem:up} for each $X^{\epsi, \kappa}$ give
\begin{align*}
&\lim_{\epsi \to 0} \bE\left[\sup_{t \in [0, T]} \left|\int_0^{t} H_s f\left(\frac{K_s X^\epsi_s}{\epsi}\right)ds - \int_0^t H_s w_s ds \right| \wedge 1 \right] \\
&\quad \leqslant \lim_{\epsi \to 0} \bE\left[\bigg(\sup_{t \in [0, T]} \left|\int_0^{t} H^\kappa_s f\left(\frac{K^\kappa_s X^{\epsi,\kappa}_s}{\epsi}\right)ds - \int_0^t H^\kappa_s w^\kappa_s ds \right| \wedge 1 \Bigg) \one_{\{\tau^\kappa = T\}} \right] + \bP[\tau^\kappa < T] \notag \\
&\quad \leqslant \bP[\tau^\kappa < T].
\end{align*}
Now Theorem \ref{theorem:up} for $X^\epsi$ follows from one of the equivalent characterizations of the convergence in probability (cf. \cite[p.~63]{kallenberg.02}) by letting $\kappa \to 0$ and using \eqref{eq:tau kappa}.

Finally, by \eqref{eq:X epsi kappa}, Theorem \ref{theorem} for each $X^{\epsi, \kappa}$, and H\"older's inequality, we obtain
\begin{align*} 
&\lim_{\epsi \to 0} \bE\left[\sup_{t \in [0, T]} \left|\int_0^{t} H_s f\left(\frac{K_s X^\epsi_s}{\epsi}\right)ds - \int_0^t H_s w_s ds \right|^p \right] \\
&\quad \leqslant \lim_{\epsi \to 0} \bE\left[\sup_{t \in [0, T]} \left|\int_0^{t} H^\kappa_s f\left(\frac{K^\kappa_s X^{\epsi,\kappa}_s}{\epsi}\right)ds - \int_0^t H^\kappa_s w^\kappa_s ds \right|^p \one_{\{\tau^\kappa = T\}} \right] \\
&\quad\quad + \sup_{\epsi > 0} \bE\left[\sup_{t \in [0, T]} \left|\int_0^{t} H_s f\left(\frac{K_s X^\epsi_s}{\epsi}\right)ds - \int_0^t H_s w_s ds \right|^p \one_{\{\tau^\kappa < T\}} \right]  \notag \\
&\quad \leqslant \sup_{\epsi > 0}  \bE\left[\left(\int_0^T H_s f\left(\frac{K_s X^\epsi_s}{\epsi}\right)ds + \int_0^T H_s w_s ds \right)^p \one_{\{\tau^\kappa < T\}} \right] \\
&\quad \leqslant \sup_{\epsi > 0} \bE\left[\left(\int_0^T H_s f\left(\frac{K_s X^\epsi_s}{\epsi}\right)ds + \int_0^T H_s w_s ds \right)^{p(1 + \eta)}\right]^{\frac{1}{1+\eta}} \bP[\tau^\kappa < T]^{\frac{\eta}{1 +\eta}},
\end{align*}
with $\eta$ as in Assumption \ref{ass.integrability}. Now Theorem \ref{theorem} for $X^\epsi$ follows by letting $\kappa \to 0$, using \eqref{eq:tau kappa} and noting that, under Assumption~\ref{ass.novikov} and~\ref{ass.integrability}, 
\begin{equation}
\label{eq:theorem:bounded estimate}
\sup_{\epsi > 0}  \bE\left[\left(\int_0^T H_s f\left(\frac{K_s X^\epsi_s}{\epsi}\right)ds + \int_0^T H_s w_s ds \right)^{p(1+ \eta)}\right] < \infty.
\end{equation}
Indeed, using the elementary inequality $(a + b)^{p(1+\eta)} \leqslant 2^{p(1+\eta)}(a^{p(1+\eta)} + b^{p(1+\eta)})$ for $a , b \geqslant 0$, \eqref{eq:theorem:bounded estimate} follows from 
\begin{equation*}
\sup_{\epsi > 0}  \bE\left[\left(\int_0^T H_s f\left(\frac{K_s X^\epsi_s}{\epsi}\right) ds\right)^{p(1+\eta)} \right] < \infty \quad \text{and} \quad \bE\left[\left(\int_0^T H_s w_s ds \right)^{p(1+ \eta)}\right] < \infty.
\end{equation*}
Here, the first estimate follows from Jensen's inequality and Lemma \ref{lem:U.I}, the second estimate follows from Lemma \ref{lem:app:integrability estimate}.

In summary, it therefore remains to establish Proposition~\ref{prop:sde} and Theorem~\ref{theorem}, respectively, for uniformly bounded coefficients in order to prove Proposition~\ref{prop:sde} and Theorem~\ref{theorem:up}. To establish Theorem~\ref{theorem} for general coefficients, it additionally remains to be shown that the integrability conditions from Assumptions~\ref{ass.novikov} and  \ref{ass.integrability}  imply Lemmas~\ref{lem:U.I} and~\ref{lem:app:integrability estimate}.

\section{Proof of Proposition~\ref{prop:sde}}
\label{section.proof.ex}

In this section, we establish that the SDE~\eqref{eq:SDE} has a unique strong solution on $[0, T]$ for each fixed $\epsi > 0$ and $x_0^\epsi \in \bR$ given that the processes $b$, $c$, $L$ and $M$ are bounded from above by $1/\kappa$ and  $c$, $L$ and $M$ are bounded from below by $\kappa$, for some $\kappa>0$.
By the localization argument from Section~\ref{sec:reduction}, this assumption is without loss of generality in the context of Proposition~\ref{prop:sde}.%\footnote{We only assume Assumption~\ref{ass.novikov} rather than boundedness of the coefficients here, since results from this section are also used in Appendices~\ref{app:UI} and \ref{app:maximal inequality}, where Assumption~\ref{ass.novikov} is in force but the coefficients are not necessarily bounded.}

First, note that it suffices to show that the SDE \eqref{eq:SDE} has a unique strong solution after a bijective time change, after which the SDE has constant volatility. To this end, for fixed $\epsi > 0$, set
\begin{equation}
\label{eq:xi.epsi.m}
\xi^\epsi=\int_{0}^{T} \epsi^{-2} c_{t} \,dt,
\end{equation}
and make the following standard observation:

\begin{lemma}
	\label{lem:time.change}
	For each $\epsi>0$, the family of stopping times (indexed by $\xi$)
	\begin{equation*}
	u^\epsi_\xi:=\begin{cases} \inf\left\{s\in\bR_+: \int_{0}^{s}\epsi^{-2} c_{r} \,dr >\xi\right\}, &\xi\leqslant\xi^\epsi,\\
T, &\xi>\xi^\epsi,
	\end{cases}
	\end{equation*}
	is strictly increasing in $\xi$ on $\left[ 0,\xi^\epsi(\omega) \right]$ for almost every $\omega\in\Omega$, forms a stochastic time change, and satisfies $u^\epsi_{\xi^\epsi} = T$. Moreover, for almost every $\omega\in\Omega$, $\xi \mapsto u^\epsi_\xi(\omega)$ is differentiable with derivative $\epsi^2/c_{u^\epsi_\xi(\omega)}(\omega)$ on $[0,\xi^\epsi(\omega)]$.
\end{lemma}

We proceed to define a time-changed Brownian motion. Set
\begin{equation}
\label{eq:What}
\widetilde{W}^\epsi_\xi=\int_{0}^{u^\epsi_\xi}\epsi^{-1}\sqrt{c_{s}}dW_s, \quad \xi\geqslant 0.
\end{equation}
This is a $\bP$-Brownian motion, stopped at $\xi^\epsi$, relative to the filtration $\bG^\epsi=(\cG^\epsi_\xi)_{\xi\in\bR_+}$ with $\cG^\epsi_\xi=\cF_{u^\epsi_\xi}$. By Lemma \ref{lem:time.change} and It\^o's formula, it suffices to show that the process $(\widetilde X^\epsi_\xi)_{\xi \geqslant 0}$, defined by
\begin{equation}
\label{eq:relation.X}
\widetilde X^\epsi_\xi := X^\epsi_{u^\epsi_\xi}/\epsi, \quad \xi \geqslant 0,
\end{equation}
is the unique strong solution of the SDE
\begin{equation}
\label{eq:SDE.proof.existence:Q}
d\widetilde{X}^\epsi_\xi = \left(\epsi \frac{b_{u^\epsi_\xi}}{c_{u^\epsi_\xi}} -L_{u^\epsi_\xi} g\left(M_{u^\epsi_\xi}\widetilde{X}^\epsi_\xi\right)\right)\one_{\{\xi \leqslant \xi^\epsi\}} d\xi+ \one_{\{\xi \leqslant \xi^\epsi\}}  d\widetilde{W}^{\epsi}_\xi, \qquad \widetilde{X}^\epsi_0=x^\epsi_0/\epsi.
\end{equation}

\medskip

Next, observe that it suffices to show that the SDE \eqref{eq:SDE.proof.existence:Q} has a unique strong solution on $\llbracket 0, \xi^\epsi \rrbracket$ (after  $\xi^\epsi$, $\widetilde X^\epsi$ is trivially constant) under a measure $\bQ$  that is equivalent to $\bP$ on $\cF_T = \cG_{\xi^\epsi}$. Set
\begin{align} 
\label{eq:proba.change.drift.Delta}
\frac{d\bQ}{d\bP} &:=\exp\left(-\int_{0}^{T}\frac{b_t}{\sqrt{c_t}}dW_t-\frac{1}{2}\int_{0}^{T}\frac{\left(b_t\right)^2}{c_t}dt\right) 
:=\exp\Bigg(-\int_{0}^{\xi^\epsi}\epsi\frac{b_{u^\epsi_\xi}}{c_{u^\epsi_\xi}}d\widetilde{W}^{\epsi}_\xi-\frac{1}{2}\int_{0}^{\xi^\epsi}\epsi^2\Bigg(\frac{b_{u^\epsi_\xi}}{c_{u^\epsi_\xi}}\Bigg)^2d\xi\Bigg),
\end{align}
so that the first part of the drift of \eqref{eq:SDE.proof.existence:Q} is absorbed by the corresponding change of measure. (Note that $\bQ$ is well defined by Novikov's condition given that $\int_0^t \frac{b_u^2}{c_u} du \leq \frac{1}{\kappa}$.) By Girsanov's Theorem, 
\begin{equation}
\label{eq:What:Q}
\widetilde{W}^{\epsi,\bQ}_\xi=\widetilde{W}^\epsi_\xi+\int_{0}^{\xi\wedge\xi^\epsi}\epsi\frac{b_{u^\epsi_y}}{c_{u^\epsi_{y}}}dy, \quad \xi\geqslant 0,
\end{equation}
in turn is a $\bQ$-Brownian motion, stopped at $\xi^\epsi$, relative to the filtration $\bG^\epsi$. Thus, it suffices to show that there is a unique strong solution of
\begin{equation}
\label{eq:SDE.proof.existence}
d\widetilde{X}^\epsi_\xi = -L_{u^\epsi_\xi} g\Big(M_{u^\epsi_\xi}\widetilde{X}^\epsi_\xi\Big) \one_{\left\{\xi\leqslant \xi^\epsi\right\}}d\xi+\one_{\left\{\xi\leqslant \xi^\epsi\right\}}d\widetilde{W}^{\epsi,\bQ}_\xi, \qquad \widetilde{X}^\epsi_0= x^\epsi_0/\epsi.
\end{equation}

This is established in the following result:

\begin{proposition}
\label{proposition:SDE.proof.existence}
For each $\epsi > 0$, there is a unique strong solution $\widetilde{X}^\epsi$ of the SDE \eqref{eq:SDE.proof.existence}. 
\end{proposition}

\begin{proof}
For $n\in\bN$, define the bounded function 
\begin{equation*}
g^{(n)}(x)=\text{sgn}(x)\left(\left|g(x)\right|\wedge n\right),
\end{equation*}
and consider the same SDE as above but with truncated drift,
\begin{equation}
\label{eq:truncated.SDE.Delta}
d\widetilde{X}^{\epsi,n}_\xi =-L_{u^\epsi_\xi} g^{(n)}\Big(M_{u^\epsi_\xi}\widetilde{X}^\epsi_\xi\Big) \one_{\left\{\xi\leqslant \xi^\epsi\right\}}d\xi+\one_{\left\{\xi\leqslant \xi^\epsi\right\}}d\widetilde{W}^{\epsi,\bQ}_\xi.
\end{equation}
The function $g^{(n)}$ is Lipschitz continuous with Lipschitz constant $K_n$ (because $g$ is locally Lipschitz and $\{g\leqslant n\}$ is compact). As a consequence, the (random) function $f:(\xi,\omega,x)\mapsto -L_{u^\epsi_\xi}g^{(n)}(M_{u^\epsi_\xi}x)$ satisfies 
\begin{equation*}
\left|f\left(\xi,\omega,x\right)-f\left(\xi,\omega,y\right)\right|\leqslant K_nL_{u^\epsi_\xi}(\omega)M_{u^\epsi_\xi}(\omega)\left|x-y\right|, \quad \mbox{for all $(\xi,\omega,x,y)\in \bR_+\times\Omega\times\bR^2$}.
\end{equation*}
The random variable $K = K_n\sup_{t\in\left[0,T\right]}L_tM_t$ is almost surely finite by continuity of $L$ and $M$ on $[0,T]$. Hence, \cite[Theorem V.7]{protter.05} shows that there exists a unique strong solution of the truncated SDE \eqref{eq:truncated.SDE.Delta}. Define the stopping time
\begin{equation}
\label{eq:st.time.pn}
\xi^{\epsi, n} :=\inf\big\{\xi\in [0,\xi^\epsi]: \big|M_{u^\epsi_\xi}X^{\epsi,n}_\xi\big|\geqslant g^{-1}(n)\big\}\wedge\xi^\epsi,
\end{equation} 
where $g^{-1}(n) = \inf \{x > 0: g(x) > n\}$.
On $\llbracket 0, \xi^{\epsi, n}\rrbracket$, the processes $\widetilde{X}^{\epsi}$ and $\widetilde{X}^{\epsi,n}$ satisfy the same SDE with the same initial condition and are therefore indistinguishable. The squared truncated process $(\widetilde{X}^{\epsi,n}_\xi)^2_{\xi\in\bR_+}$ has dynamics
\begin{align*}
d\big(\widetilde{X}^{\epsi,n}_\xi\big)^2 =&  \Big(1-2L_{u^\epsi_\xi}\widetilde{X}^{\epsi,n}_\xi g^{(n)}\Big(M_{u^\epsi_\xi}\widetilde{X}^{\epsi,n}_\xi\Big)\Big)\one_{\left\{\xi\leqslant \xi^\epsi\right\}}d\xi+2 \sqrt{\big(\widetilde{X}^{\epsi,n}_\xi\big)^2} \text{sgn}\Big(\widetilde{X}^{\epsi,n}_\xi\Big)\one_{\left\{\xi\leqslant \xi^\epsi\right\}}d\widetilde{W}^{\epsi,\bQ}_\xi.
\end{align*}
By L\'evy's characterisation of Brownian motion (cf.~\cite[Theorem 3.3.16]{karatzas.shreve.91}), the process
\begin{equation*}
\widetilde B^{\epsi,\bQ,n}_\xi=\int_{0}^{\xi}\text{sgn}\big(\widetilde{X}^{\epsi,n}_y\big)d\widetilde{W}^{\epsi,\bQ}_y
\end{equation*} 
is a $\bQ$-Brownian motion, stopped at time $\xi^\epsi$. Moreover, $x\mapsto xg^{(n)}(x)$ is an even function. Therefore, we can rewrite the dynamics of the squared truncated process as 
\begin{align*}
d\big(\widetilde{X}^{\epsi,n}_\xi\big)^2 =&  \bigg(1-2L_{u^\epsi_\xi}\sqrt{\big(\widetilde{X}^{\epsi,n}_\xi\big)^2}g^{(n)}\bigg(M_{u^\epsi_\xi}\sqrt{\big(\widetilde{X}^{\epsi,n}_\xi\big)^2}\bigg)\bigg)\one_{\left\{\xi\leqslant \xi^\epsi\right\}}d\xi +2 \sqrt{\big(\widetilde{X}^{\epsi,n}_\xi\big)^2}\one_{\left\{\xi\leqslant \xi^\epsi\right\}}d\widetilde{B}^{\epsi,\bQ,n}_\xi.
\end{align*}
Let $Y^{\epsi,n}$ be the unique strong solution of
\begin{equation*}
	dY^{\epsi,n}_\xi = \one_{\left\{\xi\leqslant \xi^\epsi\right\}}d\xi+2 \sqrt{Y^{\epsi,n}_\xi} \one_{\left\{\xi\leqslant \xi^\epsi\right\}} d\widetilde{B}^{\epsi,\bQ,n}_\xi, \qquad Y^{\epsi,n}_0=(x^\epsi_0)^2/\epsi^2.
\end{equation*}
This process is -- for each $n$ -- the square of a 1-dimensional Bessel process started at $(x^\epsi_0)^2/\epsi^2$ and stopped at time $\xi^\epsi$ (cf.~\cite[Definition XI.1.1]{revuz.yor.99}).  In particular, it is a submartingale that has finite moments at all bounded stopping times -- independent of $n$; see \cite[Chapter~XI]{revuz.yor.99}. The comparison theorem for SDEs in the form of Lemma \ref{lem:app:comparison} yields
\begin{equation*}
\bQ\Big[Y^{\epsi,n}_\xi\geqslant \big(\widetilde{X}^{\epsi,n}_\xi\big)^2, \text{ for all } \xi\in \bR_+ \Big] = 1.
\end{equation*}
Hence, for $y >0$ and $n \in \bN$, by the definition of $\xi^{\epsi, n}$, the above comparison argument, the bound on $M$ and Doob's maximal inequality applied to the non-negative submartingale $Y^{\epsi,n}$, we obtain
\begin{align*}
\bQ\left[\xi^{\epsi,n}< y\wedge\xi^\epsi\right] &\leqslant\bQ\left[\sup\left\{\Big|M_{u^\epsi_\xi}\widetilde{X}^{\epsi,n}_\xi\Big|:~\xi\in\left[0,y\wedge\xi^\epsi\right]\right\} > g^{-1}(n)\right]\\
&\leqslant\bQ\left[\sup\left\{\Big(\widetilde{X}^{\epsi,n}_\xi\Big)^2: \xi\in\left[0,y\wedge\xi^\epsi\right]\right\} > \kappa^{2}g^{-1}(n)^2\right]\\
&\leqslant \bQ\left[\sup\left\{Y^{\epsi,n}_{\xi}: \xi\in\left[0,y \wedge\xi^\epsi\right]\right\} > \kappa^{2}g^{-1}(n)^2\right] \leqslant \frac{\bE_{\bQ}\left[Y^{\epsi,n}_{y\wedge\xi^\epsi}\right]}{\kappa^{2}g^{-1}(n)^2} \leqslant \frac{\text{BES1}(y)}{\kappa^{2}g^{-1}(n)^2} ,
\end{align*}
where $\text{BES1}(y)$ denotes the expectation at time $y$ of the square of a $1$-dimensional Bessel process started at $(x^\epsi_0)^2/\epsi^2$. Letting $n \to \infty$ and using that $\lim_{n \to \infty} g^{-1}(n) = \infty$ shows that for arbitrary $x>0$,
\begin{equation*}
\bQ\left[\lim_{n\to\infty}\xi^{\epsi,n}\wedge x\wedge\xi^\epsi=x\wedge\xi^\epsi\right]=1.
\end{equation*}

Therefore, the solution of \eqref{eq:SDE.proof.existence} exists $\bQ$-a.s.\ on $\bR_+$ and in particular, on $\llbracket 0,\xi^\epsi\rrbracket$. The solution on $\bR_+$ is unique as it coincides $\bQ$-a.s.\ with all the solutions on the smaller interval $\llbracket 0, \xi^{\epsi, n} \rrbracket$. These solutions are unique by global existence and uniqueness for functionally Lipschitz SDEs~\cite[Theorem~V.7]{protter.05}.
\end{proof}

The next result holds for processes as in our main setting ($b$ adapted and locally bounded, $c$, $L$ and $M$ adapted, continuous and positive), provided Assumption~\ref{ass.novikov} is satisfied.

\begin{proposition}
Suppose that Assumption~\ref{ass.novikov} is satisfied. Then for each $k \geq 0$,
\begin{equation}
\label{eq:proposition:SDE.proof.existence:moment condition}
\sup_{s \in [0, t]}\bE\Big[\big|\widetilde{X}^{\epsi}_s\big|^{k}\Big]<\infty, \quad \text{for all } t \geq 0.
\end{equation}
\end{proposition}

\begin{proof}
This assertion follows by a similar argument as in the proof of Proposition~\ref{proposition:SDE.proof.existence}. Here, we compare $(\widetilde{X}^{\epsi})^2$ to the unique strong solution of the SDE
\begin{equation*}
dY^{(\epsi)}_\xi = \one_{\left\{\xi\leqslant \xi^\epsi\right\}}d\xi+2 \sqrt{Y^{(\epsi)}_\xi}\one_{\left\{\xi\leqslant \xi^\epsi\right\}}d\widetilde{B}^{\epsi,\bQ}_\xi, \qquad Y^{(\epsi)}_0=(x_0^\epsi)^2/\epsi^2,
\end{equation*}
where 
\begin{equation*}
\widetilde B^{\epsi,\bQ}_\xi=\int_{0}^{\xi}\text{sgn}\big(\widetilde{X}^{\epsi}_y\big)d\widetilde{W}^{\epsi,\bQ}_y
\end{equation*} 
is a $\bQ$-Brownian motion stopped at $\xi^\epsi$. We then use Doob's maximal inequality and that the square of the 1-dimensional Bessel process is a submartingale and has finite moments of all orders at all finite times \cite[Chapter~XI]{revuz.yor.99}. Combined with H\"older inequality and Assumption~\ref{ass.novikov}, this in turn yields \eqref{eq:proposition:SDE.proof.existence:moment condition} and thereby completes the proof.
\end{proof}

\section{Proofs of Theorems~\ref{theorem:up} and \ref{theorem} for Bounded Coefficients}
\label{sec:proofs}

We now turn to the proof of our main results, Theorems~\ref{theorem:up} and \ref{theorem}. By the localization argument from Section~\ref{sec:reduction}, we can and will assume throughout without loss of generality that there is $\kappa \in (0,1)$ such that 
\begin{equation}
\label{eq:ass kappa}
\int_0^T \frac{b_u^2}{c_u} du \leqslant \frac{1}{\kappa}, H_t \leqslant \frac{1}{\kappa}, K_t \leqslant \frac{1}{\kappa}, c_t  \in \left[\kappa, \frac{1}{\kappa}\right], L_t \in \left[\kappa, \frac{1}{\kappa}\right], M_t \in \left[\kappa, \frac{1}{\kappa}\right], \quad t\in [0, T].
\end{equation}
Under \eqref{eq:ass kappa}, Assumptions~\ref{ass.novikov} and \ref{ass.integrability} are trivially satisfied. Thus we can appeal to the results from the appendices for processes with such bounded coefficients and, after this localization, Theorem~\ref{theorem:up} is simply a corollary of Theorem~\ref{theorem} because $\mathcal{S}^p$ convergence implies uniform convergence in probability. In order to prove Theorem~\ref{theorem:up} for general coefficients, it therefore remains to establish Theorem~\ref{theorem} with bounded coefficients as in \eqref{eq:ass kappa}. To complete the proof of Theorem~\ref{theorem} for general coefficients, it additionally remains to establish Lemma~\ref{lem:U.I} to ensure the uniform integrability required for the localization argument from Section~\ref{sec:reduction}, and Lemma~\ref{lem:app:integrability estimate} to ensure the integrability of the limit.

Furthermore, we assume in the following that the function $f$ is nondecreasing on $\bR_+$. This is without loss of generality: since $f$ is of finite variation on every compact of $\bR$ and even, it can be written as the difference of two non-decreasing and even functions, $f_1$ and $f_2$. Then, applying Theorem~\ref{theorem:up} and \ref{theorem} to $f_1$ and $f_2$ yields the results for $f$. Finally, the result still holds after addition of a constant to $f$; we will therefore assume without loss of generality that $f$ is positive on $\bR$.

\subsection{Local estimation} 
\label{sec:local estimation}
We start by estimating the integral on the left-hand side of \eqref{eq:thm:up} ``locally'', i.e., on the intervals $[t, t+\epsi]$ for $t \in [0, T)$ and $\epsi > 0$ sufficiently small.\footnote{Note that the length of the interval could alternatively be taken equal to $\epsi^{r}$, for any $r\in(0,2)$, in which case \eqref{eq:local estimate} is modified to $\epsi^{-r}\int_{t}^{t+\epsi^r} H_s f \bigg(\frac{K_s X^\epsi_s}{\epsi} \bigg) ds$. We choose $r=1$ to obtain the simplest formulas.} More precisely, we study the limit of the normalised integral
\begin{equation}
\label{eq:local estimate}
\epsi^{-1}\int_{t}^{t+\epsi} H_s f \bigg(\frac{K_s X^\epsi_s}{\epsi} \bigg) ds.
\end{equation}
To this end, we proceed in three steps. First, we rescale and time change the process $X^\epsi$ as in Section~\ref{section.proof.ex} and also use some stopping arguments to bound the integral \eqref{eq:local estimate} from above and from below by expressions only involving the rescaled and time-changed process $\widetilde X^\epsi$ from \eqref{eq:relation.X} and $\cF_t$-measurable random variables. In a second step we approximate those $\cF_t$-measurable random variables by elementary random variables. In a final step, we use ergodic theorems for one-dimensional diffusions to compute the ``local limits'' \eqref{eq:local estimate}.
\medskip

\paragraph{Step 1: Stopping and time change.} 
In order to keep the ``slowly-varying'' processes $c$, $H$, $K$, $L$, and $M$ in a small interval around their values at time $t$, we define for fixed $\epsi \in (0, \frac{\kappa^2(T-t)}{4})$
and $\delta \in (0, \tfrac{\kappa}{2})$ the stopping time
\begin{align}
\tau^{\epsi,\delta}_t &= \inf\bigg\{s\in\left[t,t+\epsi\right]:~ \frac{H_s}{c_s}\notin\bigg[\frac{H_t(1-\delta)}{c_t},\frac{H_t + \delta}{c_t}\bigg]\!, c_s\notin\left[c_t(1-\delta), c_t(1+\delta)\right], \notag\\
&\quad L_s\notin\left[L_t - \delta,L_t+\delta\right]\!, M_s\notin\left[M_t - \delta,M_t+\delta\right]\!, K_s\notin\left[K_t (1-\delta),K_t+\delta\right] \bigg\}\wedge\left(t+\epsi\right). \label{tau.epsi.st.time}
\end{align}
By uniform continuity of $c, H, K, L, M$ on $[0,T]$, there exists a random variable $\epsi_{\delta}>0$ (which is independent of $t$) such that, for $\epsi \in (0, \frac{\kappa^2(T-t)}{4})$, we have  
\begin{equation*}
\tau^{\epsi,\delta}_t= t+ \epsi \quad \mbox{on $\{0 < \epsi \leqslant \epsi_{\delta}\}$}.
\end{equation*}
It follows that for $\epsi \in (0, \frac{\kappa^2(T-t)}{4})$,
\begin{align}
\label{stopped.cost.tau.1}
\epsi^{-1}\int_{t}^{t + \epsi}H_sf\left(\frac{K_s X^\epsi_{s}}{\epsi}\right)ds=\epsi^{-1}\int_{t}^{\tau^{\epsi,\delta}_t}H_sf\left(\frac{K_s X^\epsi_{s}}{\epsi}\right)ds \quad \text{on } \{0 < \epsi \leqslant \epsi_\delta\}.
\end{align}
We proceed to study the integral on the right-hand side of \eqref{stopped.cost.tau.1}. To this end, we pass to time-changed quantities as in Section \ref{section.proof.ex}, with the difference that we start time at $t$. So set
\begin{align*}
\xi^{\epsi}_t = \int_{t}^{T} \epsi^{-2} c_s ds, \quad \xi^{\epsi,\delta}_t &= \int_{t}^{\tau^{\epsi,\delta}_t} \epsi^{-2} c_s ds, \quad u_\xi^{\epsi,t} = u^\epsi_{\xi+\int_{0}^{t} \epsi^{-2} c_s ds}, \quad \widetilde{X}^{\epsi,t}_{\xi} = X^\epsi_{u_\xi^{\epsi,t}}/\epsi.
\end{align*}  
Here, $\xi^{\epsi}_t $ and $\xi^{\epsi,\delta}_t$ denote the lengths of the intervals $[t, T]$ and $[t, \tau^{\epsi,\delta}_t]$ after the time change,  $u_\xi^{\epsi,t}$ is the family of stopping times introduced in Lemma \ref{lem:time.change}, shifted to start at the time change of $t$, and $\widetilde{X}^{\epsi,t}$ denotes the rescaled and time-changed process $\widetilde{X}^\epsi$ restarted at the time change of $t$.
Note that for $\epsi \in (0, \frac{\kappa^2(T-t)}{4})$, 
\begin{equation}
\label{eq:bound:xi epsi delta stopping}
\epsi^{-1} c_t (1 - \delta)  \leqslant \xi^{\epsi,\delta}_t  \leqslant \epsi^{-1}  c_t (1 + \delta) \quad \text{on } \{0 < \epsi \leqslant \epsi_{\delta}\}.
\end{equation}
Moreover, note that since $c \in [\kappa, \kappa^{-1}]$ and $\delta\in(0,1)$, we have $\xi^\epsi_t \geqslant \epsi^{-2} \kappa (T - t)$ and $2\epsi^{-1} c_t (1 + \delta) \leqslant 4 \epsi^{-1}/\kappa$. Together with $\kappa(T-t) \geqslant 4 \kappa^{-1} \epsi$, this yields
\begin{equation}
\label{eq:bound:xi epsi delta global}
2\epsi^{-1}  c_t (1 + \delta) \leqslant \xi^{\epsi}_t.
\end{equation}
Hence, even though $\epsi^{-1}  c_t (1 + \delta)$ might be larger than $\xi^{\epsi, \delta}_t$ with positive probability, it is always smaller than the remaining time to the time horizon $\xi^\epsi_t$ after the time change. By considering $\widetilde{X}^{\epsi,t}$ on the interval $[0, \xi^{\epsi,\delta}_t]$, we can now separate the quickly-oscillating displacement from the other, more slowly-varying processes in the estimation:

\begin{lemma}
\label{lem:sandwich}
Let $t \in [0, T)$, $\epsi \in (0, \frac{\kappa^2(T-t)}{4})$ and $\delta \in (0, \tfrac{\kappa}{2})$. Then:
\begin{equation}
\begin{split}
\one_{\{\epsi \leqslant \epsi_\delta\}}&\left(1-\delta\right)^2 H_t\Bigg(\frac{1}{\epsi^{-1} c_t(1 - \delta)}\int_{0}^{\epsi^{-1} c_t (1 - \delta)}f\left(K_t (1 - \delta) \widetilde{X}^{\epsi,t}_{\xi}\right) d\xi\Bigg) \\ 
&\leqslant \one_{\{\epsi \leqslant \epsi_\delta\}} \epsi^{-1}\int_{t}^{t+\epsi}H_rf\bigg(\frac{K_r X^\epsi_{r}}{\epsi}\bigg)dr \\
&\leqslant \one_{\{\epsi \leqslant \epsi_\delta\}} \left(1+\delta\right)(H_t +\delta)\Bigg(\frac{1}{\epsi^{-1} c_t(1 + \delta)}\int_{0}^{\epsi^{-1} c_t(1 +\delta)}f\left((K_t + \delta)\widetilde{X}^{\epsi,t}_{\xi}\right) d\xi\Bigg).
\end{split}
\label{eq:sandwich1bis}
\end{equation}
\end{lemma}

\begin{proof}
On $\{0<\epsi\leqslant\epsi_{\delta}\}$,  \eqref{stopped.cost.tau.1} and the time change $s=u^{\epsi,t}_\xi$ give
\begin{align*}
\epsi^{-1}\int_{t}^{t+\epsi}H_s f\left(\frac{K_s X^\epsi_{s}}{\epsi}\right)ds= \epsi^{-1}\int_{t}^{\tau^{\epsi,\delta}_t}H_s f\left(\frac{K_s X^\epsi_{s}}{\epsi}\right)  dr=\epsi\int_{0}^{\xi^{\epsi,\delta}_t}\frac{H_{u^{\epsi,t}_\xi}}{c_{u^{\epsi,t}_\xi}}f\Big(K_{u^{\epsi,t}_\xi}\widetilde{X}^{\epsi,t}_{\xi}\Big)d\xi.
\end{align*}
Now, \eqref{eq:sandwich1bis} follows by using that, by definition of $\tau_t^{\epsi,\delta}$, 
\begin{equation*}
\frac{H_{u^{\epsi,t}_\xi}}{c_{u^{\epsi,t}_\xi}} \in\bigg[(1-\delta)\frac{H_t}{c_t},\frac{H_t + \delta}{c_t}\bigg], \quad \text{and} \quad K_{u^{\epsi,t}_\xi} \in [K_t(1-\delta), K_t + \delta], 
\quad \xi \in [0, \xi^{\epsi, \delta}_t],
\end{equation*}
and also taking into account \eqref{eq:bound:xi epsi delta stopping}, and that the function $f$ is even and non decreasing on $\bR_+$.
\end{proof}

\paragraph{Step 2: Approximation by elementary random variables.}  We now turn to the estimation of the terms that appear in the bounds from Lemma~\ref{lem:sandwich}. To this end, we approximate $K_t$ and $c_t$ by elementary $\cF_t$-measurable random variables and the SDE for $\widetilde{X}^{\epsi,t}$ by an SDE with coefficients that are constant over time and elementary $\cF_t$-measurable random variables.

To this end, for 
$n \in \bN$ with $n \geqslant\frac{2}{\kappa} $, $\Upsilon_t \in \{c_t, L_t, M_t, K_t\}$,  and $i \in \bN$ set
\begin{equation}
\label{eq:Omega:Upsilon}
\Omega^{i, \Upsilon_t,n} := 
\left\{\frac{i}{n}\leqslant \Upsilon_t<\frac{(i+1)}{n}\right\},
\end{equation}
and define the random variables $c_t^{n,+}, c_t^{n,-}, L_t^{n,+}, L_t^{n,-}, M_t^{n,+}, M_t^{n,-}, K_t^{n,+}, K_t^{n,-}$ for $n \in \bN$ with $n \geqslant \frac{2}{\kappa} $ by 
\begin{align}\label{eq:proc.approx}
\Upsilon^{n,+}_t =\sum_{i=0}^{\infty}\frac{i+1}{n} \one_{\{\Omega^{i,\Upsilon_t,n}\}} \quad \text{and} \quad \Upsilon^{n,-}_t = \sum_{i=0}^{\infty}\frac{i}{n}\one_{\{\Omega^{i,\Upsilon_t,n}\}},
\end{align}
where $\Upsilon \in \{c, L, M, K\}$. Note that, for fixed $n \geqslant \frac{2}{\kappa}$, and for $\delta\in\left(0,\frac{\kappa}{2}\right)$,
\begin{equation}
\label{eq:Omega:Upsilon:condition}
\text{$\bP\left[\Omega^{i, \Upsilon_t,n}\right] = 0$ for $\Upsilon_t \in \{c_t, L_t, M_t\}$ and $i \in \bN$ with $\frac{i}{n} \leqslant \delta$}
\end{equation}
because $c, L, M \geqslant \kappa$.\footnote{Note that $P[\Omega^{0, K_t,n}]$ may be positive as $K$ is only nonnegative.} This implies that $c_t^{n,+}, c_t^{n,-}, L_t^{n,+}, L_t^{n,-}, M_t^{n,+}, M_t^{n,-} > \delta$.
Moreover, note that for each $n \in \bN$ with $n \geqslant \frac{2}{\kappa} $ by the fact that $c_t \geqslant \kappa$,
\begin{equation*}
c^{n,+}_t \leqslant c_t + \frac{1}{n} \leqslant c_t + \frac{\kappa}{2} \leqslant 2 c_t.
\end{equation*}
Together with \eqref{eq:bound:xi epsi delta global}, this yields the important estimate
\begin{equation}
\label{eq:bound:xi epsi delta global:approx}
\epsi^{-1} c^{n,+}_t (1 + \delta) \leqslant \xi^{\epsi}_t.
\end{equation}
By construction,
\begin{equation}
\label{eq:Upsilon:approximation}
\Upsilon^{n,-}_t \leqslant \Upsilon_t \leqslant \Upsilon^{n,+}_t, \quad n \geqslant \frac{2}{\kappa}, \quad \text{and} \quad \lim_{n \to \infty} \Upsilon^{n,-}_t = \Upsilon_t = \lim_{n \to \infty} \Upsilon^{n,+}_t, \quad \Upsilon \in \{c, L, M, K\},
\end{equation}
and we have for $\Upsilon_t \in \{c_t, L_t, M_t\}$,
\begin{equation}
\label{eq:union.Omegas}
	\bigcup_{i=1}^{\infty}\Omega^{i, \Upsilon_t,n} =\Omega~~\mbox{ and }~~\bigcup_{i=0}^{\infty}\Omega^{i, K_t,n} =\Omega.
\end{equation}

We proceed to approximate $\widetilde{X}^{\epsi,t}$ (or more precisely $(\widetilde{X}^{\epsi,t})^2$). It follows from \eqref{eq:SDE.proof.existence} that the process $\widetilde{X}^{\epsi,t}$ satisfies on $\llbracket 0,\xi^{\epsi}_t\rrbracket$ the SDE
\begin{equation*}
d\widetilde{X}^{\epsi,t}_\xi = -L_{u^{\epsi,t}_\xi}g\Big(M_{u^{\epsi,t}_\xi}\widetilde{X}^{\epsi,t}_\xi\Big) d\xi+d\widetilde{W}^{\epsi,\bQ,t}_\xi, \quad \widetilde{X}^{\epsi,t}_0 = X^\epsi_t/\epsi,
\end{equation*}
where $\widetilde{W}^{\epsi,\bQ,t}_\xi := \widetilde{W}^{\epsi, \bQ}_{\xi+\xi^{\epsi,t}_0} - \widetilde{W}^{\epsi,  \bQ}_{\xi^{\epsi,t}_0} $ is the $\bQ$-Brownian motion from \eqref{eq:What:Q} restarted at $\xi^{\epsi,t}_0 =\int_0^t \epsi^{-2} c_s ds\leqslant \xi^\epsi$. 
Define the process $B^{\epsi,\bQ,t}$ by 
\begin{equation*}
B^{\epsi,\bQ,t}_\xi=\int_{0}^{\xi}\text{sgn}(\widetilde{X}^{\epsi,t}_y)d\widetilde{W}^{\epsi,\bQ,t}_y, \quad \xi\geqslant 0.
\end{equation*}
By L\'evy's characterisation \cite[Theorem 3.3.16]{karatzas.shreve.91}, this is a Brownian motion, stopped at $\xi^\epsi_t$. As the function $g$ is odd, It\^o's formula gives
\begin{align*}
d\Big(\widetilde{X}^{\epsi,t}_\xi\Big)^2 =&  \left(1-2L_{u^{\epsi,t}_\xi}\sqrt{\Big(\widetilde{X}^{\epsi,t}_\xi\Big)^2} g\left(M_{u^{\epsi,t}_\xi}\sqrt{\Big(\widetilde{X}^{\epsi,t}_\xi\Big)^2}\right)\right) \one_{\{\xi\leqslant\xi^{\epsi}_t\}}d\xi+2\sqrt{\Big(\widetilde{X}^{\epsi,t}_\xi\Big)^2} \one_{\{\xi\leqslant\xi^{\epsi}_t\}} dB^{\epsi,\bQ,t}_\xi.
\end{align*}
To bound $(\widetilde{X}^{\epsi,t}_\xi)^2$ from above and from below, we proceed as follows. For $\epsi > 0$, $\delta > 0$ and constants $l, m > \delta$ (which are independent of $\epsi$) and an $\cF_t$-measurable initial value $y \geqslant 0$ (which may depend on $\epsi$), let $Y^{y,l,m,\epsi,\delta,+}$, and $Y^{y,l,m,\epsi,\delta,-}$ be the unique strong solutions of the following two SDEs:
\begin{align}
dY^{y,l,m,\epsi,\delta,\pm}_\xi &= \left(1-2\left(l \mp \delta\right)\sqrt{Y^{y,l,m,\epsi,\delta, \pm}_\xi}g\left(\left(m \mp \delta\right)\sqrt{Y^{y,l,m,\epsi,\delta,\pm}_\xi}\right)\right) \one_{\{\xi\leqslant\xi^{\epsi}_t\}}  d\xi \notag \\ 
&\qquad\qquad +2\sqrt{Y^{l,m,\epsi,\delta,\pm}_\xi} \one_{\{\xi\leqslant\xi^{\epsi}_t\}} dB^{\epsi,\bQ,t}_\xi, \quad Y^{y,l,m,\epsi,\delta,\pm}_0 =y \label{sde_y.epsilon.-}.
\end{align}
Existence and uniqueness of strong solutions for \eqref{sde_y.epsilon.-} follows from Lemma \ref{lem:app:existence:SDE}.\footnote{More precisely, the solution of \eqref{sde_y.epsilon.-} corresponds to a solution of \eqref{eq:lem:app:existence:SDE}, stopped at time $\xi^\epsi_t$.}
Note that the SDEs \eqref{sde_y.epsilon.-} depend on $\epsi$ only via the Brownian motion $B^{\epsi,\bQ,t}$ and their starting value $y$.

Moreover, for an $\cF_t$-measurable random variable $y \geqslant 0$ and each $n \in \bN\setminus \{0\}$ define the continuous semimartingales $(Y^{y,L^{n,+}_t,M^{n,+}_t,\epsi,\delta,-}_\xi)_{\xi \geq0}$ and $(Y^{y,L^{n,-}_t,M^{n,-}_t,\epsi,\delta,+}_\xi)_{\xi \geq0}$ by
\begin{align}
\label{eq:Y lower}
Y^{y,L^{n,+}_t,M^{n,+}_t,\epsi,\delta,-}_\xi &= \sum_{i \geqslant 0} \sum_{j \geqslant 0} Y^{y,(i+1)/n,(j+1)/n,\epsi,\delta,-}_\xi \one_{\Omega^{i,L_t,n}}\one_{\Omega^{j,M_t,n}}, \\
Y^{y,L^{n,-}_t,M^{n,-}_t,\epsi,\delta,+}_\xi &= \sum_{i \geqslant 0} \sum_{j \geqslant 0} Y^{y,i/n,j/n,\epsi,\delta,+}_\xi \one_{\Omega^{i,L_t,n}}\one_{\Omega^{j,M_t,n}}.
\label{eq:Y upper}
\end{align}

Now \eqref{eq:Upsilon:approximation}, the definition of $\tau^{\epsi, \delta}_t$ in \eqref{tau.epsi.st.time}, the assumption that $g$ is non-decreasing and the comparison theorem for SDEs in the form of Lemma \ref{lem:app:comparison} give, for each $n \in \bN$ with $n \geqslant \frac{2}{\kappa}$ and all $i, j\in\bN$,
\begin{align*}
&\bP\bigg[Y^{Y^\epsi_0,(i+1)/n,(j+1)/n,\epsi,\delta,-}_\xi \one_{\Omega^{i,L_t,n}} \one_{\Omega^{j,M_t,n}} \leqslant \left(\widetilde{X}^{\epsi,t}_{\xi}\right)^2 \one_{\Omega^{i,L_t,n}} \one_{\Omega^{j,M_t,n}} \\
&\qquad \qquad\qquad\leqslant Y^{Y^\epsi_0,i/n,j/n,\epsi,\delta,+}_\xi  \one_{\Omega^{i,L_t,n}} \one_{\Omega^{j,M_t,n}} \text{ for all } 0\leqslant\xi \leqslant \xi^{\epsi,\delta}_t \bigg] = 1,
\end{align*}
where 
\begin{equation*}
Y^\epsi_0 := \left(\widetilde{X}^{\epsi,t}_0\right)^2 = (X^\epsi_t)^2/\epsi^2.
\end{equation*}
Together with \eqref{eq:Y lower} and \eqref{eq:Y upper}, \eqref{eq:Omega:Upsilon}, \eqref{eq:Omega:Upsilon:condition} and \eqref{eq:union.Omegas}, this yields
\begin{align}
\label{eq:comparison}
\bP\left[Y^{Y^\epsi_0,L^{n,+}_t,M^{n,+}_t,\epsi,\delta,-}_\xi\leqslant \left(\widetilde{X}^{\epsi,t}_{\xi}\right)^2 \leqslant Y^{Y^\epsi_0,L^{n,-}_t,M^{n,-}_t,\epsi,\delta,+}_\xi \text{ for all } 0\leqslant\xi\leqslant\xi^{\epsi,\delta}_\xi \right] = 1.
\end{align}

To simplify the notation in the subsequent results, define for constants $c, k, l, m$ (independent of $\epsi$) with $c, l, m > \delta$ and an $\cF_t$-valued random variable $y$ (which may depend on $\epsi$) the following two random variables:
\begin{align*}
v^{\epsi, \delta, +}_y(c,k,l,m) &:= \frac{1}{\epsi^{-1} c (1 + \delta)}\int_{0}^{\epsi^{-1} c (1 + \delta)}f\left((k + \delta) \sqrt{Y^{y,l,m,\epsi,\delta,+}_\xi}\right) d\xi, \\
v^{\epsi, \delta, -}_y(c,k,l,m) &:= \frac{1}{\epsi^{-1} c (1 - \delta)}\int_{0}^{\epsi^{-1} c (1 - \delta)}f\left(k (1 - \delta) \sqrt{Y^{y,l,m,\epsi,\delta,-}_\xi}\right) d\xi.
\end{align*}
%\texttt{Note that the plus and minus versions are NOT symmetric. This is again because K is nonnegative only.}
Note that by \eqref{eq:bound:xi epsi delta global}, $v^{\epsi, \delta, +}_y(c,k,l,m)$ and $v^{\epsi, \delta, -}_y(c,k,l,m)$  are $\cF_T$-measurable for $c\leqslant 1/\kappa$. Moreover, by comparison of SDEs in their initial values, they are non decreasing in $y$ (see Lemma~\ref{lem:app:comparison} and the proof of Lemma~\ref{lem:app:existence:SDE}). 

Now combining \eqref{eq:comparison} and \eqref{eq:Upsilon:approximation} with Lemma~\ref{lem:sandwich} and the fact that $f$ is even and nondecreasing on $\bR_+$ yields the following result.\footnote{Note that compared to Lemma~\ref{lem:sandwich}, the processes in the upper and lower bounds are replaced by the simpler approximating diffusions introduced in (\ref{sde_y.epsilon.-}) here and the frozen coefficients are approximated by finitely many values, as $c, K, L$ and $M$ are bounded from above.}

\begin{lemma}
	\label{lem:sandwich:2}
	Let $t \in [0, T)$, $\epsi \in (0, \frac{\kappa^2(T-t)}{4})$, $\delta \in (0, \tfrac{\kappa}{2})$, and $n \in \bN \setminus \{0\}$. Then:
	\begin{equation}	\label{eq:sandwich1}
	\begin{split}
&\one_{\{\epsi \leqslant \epsi_\delta\}}\left(1-\delta\right)^2 H_t \frac{c^{n,-}_t}{c^{n,+}_t} v^{\epsi, \delta, -}_{Y^\epsi_0}(c^{n,-}_t,K^{n,-}_t,L^{n,+}_t,M^{n,+}_t)   \\ 
	&\qquad \leqslant \one_{\{\epsi \leqslant\epsi_\delta\}} \epsi^{-1}\int_{t}^{t+\epsi}H_rf\left(\frac{K_r X^\epsi_{r}}{\epsi}\right)dr \\
	&\qquad \leqslant \one_{\{\epsi \leqslant \epsi_\delta\}} \left(1+\delta\right)(H_t +\delta)\frac{c^{n,+}_t}{c^{n,-}_t} v^{\epsi, \delta, +}_{Y^\epsi_0}(c^{n,+}_t,K^{n,+}_t,L^{n,-}_t,M^{n,-}_t),
	\end{split}
	\end{equation}
	where $c^{n,+}$, $c^{n,-}$, $K^{n,+}$, $K^{n,- }$, $L^{n,+}$, $L^{n,- }$, $M^{n,+}$, $M^{n,- }$ are defined as in \eqref{eq:proc.approx}.
\end{lemma}

\paragraph{Step 3: Limit theorems.}  We now combine the sandwiching inequalities from Lemma \ref{lem:sandwich:2} with an ergodic theorem for one-dimensional diffusions (cf.~Lemma~\ref{lemma.limit.epsi.ergodic}) to calculate the following ``local'' scaling limit:
 \begin{equation*}
\lim_{\epsi \to  0} \epsi^{-1}\int_{t}^{t+\epsi}H_rf\left(\frac{K_r X^\epsi_{r}}{\epsi}\right)dr, \quad t \in [0,T).
 \end{equation*}

To this end, we first establish an ergodic result, which is non-standard in that both the time horizon and the underlying process change with the small parameter at hand. As a consequence, the ergodic limit only holds in probability here rather than almost surely. To formulate this result, define for constants $k \geqslant 0$ and $l, m > \delta > 0$, the two continuous functions
\begin{align*}
w^{\delta,+}(k,l,m) &:=	\frac{\int_{\bR_+}f\left(\frac{(k  +\delta)}{m - \delta}x\right)\exp\left(-2\frac{l - \delta}{m - \delta}G\left(x\right)\right)dx}{\int_{\bR_+}\exp\left(-2\frac{l - \delta}{m - \delta}G\left(x\right)\right)dx}, \\
w^{\delta,-}(k,l,m) &:=	\frac{\int_{\bR_+}f\left(\frac{k(1 -\delta)}{m + \delta}x\right)\exp\left(-2\frac{l + \delta}{m + \delta}G\left(x\right)\right)dx}{\int_{\bR_+}\exp\left(-2\frac{l + \delta}{m + \delta}G\left(x\right)\right)dx}.
\end{align*}
%\texttt{Again the plus and minus versions are NOT symmetric. This is again because K is nonnegative only.}
With this notation, our ergodic result reads as follows:

\begin{lemma}
\label{lem:ergodic elementary}
Let $t \in [0, T)$, $\delta \in (0, \tfrac{\kappa}{2})$, and $n \in \bN$ with $n \geqslant \frac{2}{\kappa}$ be fixed. Then the following two limits hold in probability:
\begin{align}
\lim_{\epsi \to 0} v^{\epsi, \delta, \pm}_{Y^\epsi_0}(c^{n,\pm}_t,K^{n,\pm}_t,L^{n,\mp}_t,M^{n,\mp}_t)   = w^{\delta,\pm}(K^{n,\pm}_t, L^{n,\mp}_t, M^{n,\mp}_t).
\label{eq:lem:ergodic elementary:-}
\end{align}
\end{lemma}

\begin{proof}
We only spell out the argument for the ``$-$''-limit in \eqref{eq:lem:ergodic elementary:-}; the ``$+$''-limit is established analogously. By one of the equivalent characterizations of convergence in probability, we have to show that 
\begin{equation*}
\lim_{\epsi \to  0}\bE \left[\left|v^{\epsi, \delta, -}_{Y^\epsi_0}(c^{n,-}_t,K^{n,-}_t,L^{n,+}_t,M^{n,+}_t) -  w^{\delta,-}(K^{n,-}_t, L^{n,+}_t, M^{n,+}_t)\right| \wedge 1 \right] = 0.
\end{equation*}
Since both $v^{\epsi, \delta, -}_{Y^\epsi_0}(c^{n,-}_t,K^{n,-}_t,L^{n,+}_t,M^{n,+}_t)$ and $w^{\delta,-}(K^{n,-}_t, L^{n,+}_t, M^{n,+}_t)$ are $\cF_T$-measurable and the measure $\bQ$ defined in \eqref{eq:proba.change.drift.Delta} is equivalent to $\bP$ on $\cF_T$, it suffices to show that
\begin{equation*}
\lim_{\epsi \to  0}\bE_{\bQ} \left[\left|v^{\epsi, \delta, -}_{Y^\epsi_0}(c^{n,-}_t,K^{n,-}_t,L^{n,+}_t,M^{n,+}_t) -  w^{\delta,-}(K^{n,-}_t, L^{n,+}_t, M^{n,+}_t)\right| \wedge 1 \right] = 0.
\end{equation*}
For $i_c, i_K, i_L,i_M \in \bN$ set 
\begin{equation*}
\Omega_{i_c,i_K,i_L,i_M} := \Omega^{i_c,c_t,n} \cap \Omega^{i_K,K_t,n} \cap \Omega^{i_L,L_t,n} \cap \Omega^{i_M,M_t,n}, 
\end{equation*}
where, for $\Upsilon_t \in \{c_t,K_t,L_t,M_t\}$ and $i \in \{i_c,i_K,i_L,i_M\}$, the set $\Omega^{i, \Upsilon_t, n}$ is defined as in \eqref{eq:Omega:Upsilon}. By dominated convergence and \eqref{eq:Omega:Upsilon} it suffices to show that 
\begin{equation*}
\lim_{\epsi \to  0}\bE_{\bQ} \left[\left(\big|v^{\epsi, \delta, -}_{Y^\epsi_0}(c^{n,-}_t,K^{n,-}_t,L^{n,+}_t,M^{n,+}_t) -  w^{\delta,-}(K^{n,-}_t, L^{n,+}_t, M^{n,+}_t)\big| \wedge 1 \right)\one_{\Omega_{i_c,i_K,i_L,i_M} }\right] = 0,
\end{equation*}
for all $i_c, i_K, i_L,i_M \in \bN$. So fix $i_c, i_K, i_L,i_M \in \bN$ with $\bQ[\Omega_{i_c,i_K,i_L,i_M} ] > 0$. Note that by \eqref{eq:Omega:Upsilon:condition} and $\bP\approx\bQ$, this implies in particular that $i_c/n, i_L/n,i_M/n > \delta$. Using that $c^{n,-}_t,K^{n,-}_t,L^{n,+}_t,M^{n,+}_t$ take the constant values $i_c/n , i_K/n, (i_L+1)/n,  (i_M+1)/n$ on $\Omega_{i_c,i_K,i_L,i_M}$ we have, for each fixed $\epsi > 0$,
\begin{align}
&\bE_{\bQ} \left[\left(\big|v^{\epsi, \delta, -}_{Y^\epsi_0}(c^{n,-}_t,K^{n,-}_t,L^{n,+}_t,M^{n,+}_t) -  w^{\delta,-}(K^{n,-}_t, L^{n,+}_t, M^{n,+}_t)\big| \wedge 1 \right)\one_{\Omega_{i_c,i_K,i_L,i_M} }\right]  \notag \\
&\qquad =\bE_{\bQ} \Big[\Big(\Big|v^{\epsi, \delta, -}_{Y^\epsi_0}\Big(\frac{i_c}{n} , \frac{i_K}{n}, \frac{i_L+1}{n},  \frac{i_M+1}{n}\Big)-  w^{\delta,-}\left(\frac{i_c}{n} , \frac{i_K}{n}, \frac{i_L+1}{n},  \frac{i_M+1}{n}\right)\Big| \wedge 1\Big)\one_{\Omega_{i_c,i_K,i_L,i_M} } \Big].
\label{eq:ergodic,iciKiLiM}
\end{align}
We proceed to estimate the right-hand side of \eqref{eq:ergodic,iciKiLiM}. To simplify notation, set 
\begin{equation}
\bar w := w^{\delta,-}\Big(\frac{i_c}{n}, \frac{i_K}{n}, \frac{i_L+1}{n},  \frac{i_M+1}{n}\Big)
\end{equation}
and note that this is a constant. For $\lambda > 0$, we split up the expectation to the disjoint events $\{0 \leqslant Y^\epsi_0 \leqslant \frac{1}{\lambda} \}$ and $\{Y^\epsi_0 > \frac{1}{\lambda} \}$. On the event $\{0 \leqslant Y^\epsi_0 \leqslant \frac{1}{\lambda} \}$, we use that $v^{\epsi, \delta, -}_y(c,k,l,m)$ is non decreasing in $y$ together with the elementary inequality $|z - \bar{w}| \leqslant |z_{\min} -\bar{w}| +  |z_{\max} -\bar{w}|$ for $\bar{w}\in\bR$ and $z_{\min}\leqslant z\leqslant z_{\max} \in \bR$. On the event $\{Y^\epsi_0 > \frac{1}{\lambda} \}$, we use that the random variable inside the expectation on the right-hand side of \eqref{eq:ergodic,iciKiLiM} is bounded from above by $1$. Together, this yields
\begin{align}
&\bE_{\bQ} \left[\left|v^{\epsi, \delta, -}_{Y^\epsi_0}\left(\frac{i_c}{n} , \frac{i_K}{n}, \frac{i_L+1}{n},  \frac{i_M+1}{n}\right) -  \bar w \right| \wedge 1 \right] \notag \\
&\qquad\leqslant \bE_{\bQ} \left[\left|v^{\epsi, \delta, -}_{0}\left(\frac{i_c}{n} , \frac{i_K}{n}, \frac{i_L+1}{n},  \frac{i_M+1}{n}\right) -   \bar w \right| \wedge 1 \right] \notag \\
&\qquad \quad +\bE_{\bQ} \left[\left|v^{\epsi, \delta, -}_{\frac{1}{\lambda}}\left(\frac{i_c}{n} , \frac{i_K}{n}, \frac{i_L+1}{n},  \frac{i_M+1}{n}\right) -  \bar w \right| \wedge 1 \right]+\bQ\left[Y^\epsi_0 > \frac{1}{\lambda} \right].
\label{eq:ergodic,iciKiLiM:rewritten}
\end{align}
Next, note that for $y \in \{0, \frac{1}{\lambda}\}$ independent of $\epsi$, the random variables $v^{\epsi, \delta, -}_y(i_c/n , i_K/n, (i_L+1)/n,  (i_M+1)/n)$ depend on $\epsi$ only via the Brownian motion
$B^{\bQ,\epsi, t}$ which also is the only source of stochasticity. In particular, the law of $v^{\epsi, \delta, -}_y(i_c/n , i_K/n, (i_L+1)/n,  (i_M+1)/n)$ does not depend on $\epsi$. Thus, if we replace the $\bQ$-Brownian motion $B^{\bQ,\epsi, t}$ by any other fixed Brownian motion $B$ (on some different probability space), the result does not change. Hence, we can apply the ergodic theorem for one-dimensional diffusions in the form of Lemma \ref{lemma.limit.epsi.ergodic} to conclude that the first two terms on the right-hand side of \eqref{eq:ergodic,iciKiLiM:rewritten} converge to zero as $\epsi \to 0$.\footnote{This also uses crucially the estimate \eqref{eq:bound:xi epsi delta global:approx} and the fact that $i_c/n\geqslant \kappa/2$.} Finally, by Markov's inequality and Lemma~\ref{lemma.uniform.moment.bar},\footnote{Note that all assumptions in Lemma~\ref{lemma.uniform.moment.bar} are satisfied by \eqref{eq:ass kappa}.}  there is a constant $\overline{C}_2$ independent of $\epsi$ such that
\begin{equation}
\label{eq:Y0epsi}
\bQ\left[Y^\epsi_0 > \frac{1}{\lambda} \right]= \bQ\left[\left(\widetilde X^{\epsi,t}_{0}\right)^2 > \frac{1}{\lambda} \right] = \bQ\left[\left(\frac{X^\epsi_{t}}{\epsi}\right)^2 > \frac{1}{\lambda} \right]  \leqslant \lambda \bE_{\bQ}\left[\left( \frac{X^\epsi_{t}}{\epsi}\right)^2 \right] \leqslant \lambda \overline{C}_2.
\end{equation}
The claim in turn follows by letting $\lambda$ go to zero.
\end{proof}

Sending the localization parameter $\delta$ from Lemma~\ref{lem:sandwich} to zero and the discretization parameter $n$ from \eqref{eq:Omega:Upsilon} to infinity, we now obtain the following scaling limit:

\begin{proposition}
	\label{proposition.limit.integrals}
For $t \in [0, T)$, the following limit holds in probability:
	\begin{align*}
\lim_{\epsi \to  0} \epsi^{-1}\int_{t}^{t+\epsi}H_rf\left(\frac{K_r X^\epsi_{r}}{\epsi}\right)dr &=H_t \frac{\int_{\bR}f\left(\frac{K_t}{M_t}x\right) \exp\left(-2\frac{L_t}{M_t}G\left(x\right)\right)dx}{\int_{\bR}\exp\left( -2\frac{L_t}{M_t}G\left(x\right)\right)dx}.
	\end{align*}
\end{proposition}

\begin{proof} Fix $t \in [0, T)$. For $\epsi \in (0, \frac{\kappa^2(T-t)}{2})$, $\delta \in (0, \frac{\kappa}{2})$, and $n \in \bN$ with $n  \geqslant \frac{2}{\kappa}$, set
\begin{align*}
a^{\epsi, \delta, n, -}_t &:= \one_{\{\epsi \leqslant \epsi_\delta\}} \left(1-\delta\right)^2 H_t \frac{c^{n,-}_t }{c^{n,+}_t } v^{\epsi, \delta, -}_{Y^\epsi_0}(c^{n,-}_t,K^{n,-}_t,L^{n,+}_t,M^{n,+}_t),
\\
a^{\epsi, \delta, n, +}_t &:= \one_{\{\epsi \leqslant \epsi_\delta\}}(1+\delta)(H_t + \delta) \frac{c^{n,+}_t}{c^{n,-}_t} v^{\epsi, \delta, +}_{Y^\epsi_0}(c^{n,+}_t,K^{n,+}_t,L^{n,-}_t,M^{n,-}_t)  , \\
a^{1,\epsi,\delta}_t &:= 
\one_{\{\epsi \leqslant \epsi_\delta\}} \epsi^{-1}\int_{t}^{t+\epsi}H_sf\left(\frac{K_s X^\epsi_{s}}{\epsi}\right)ds, \\
a^{1,\epsi}_t &:= \epsi^{-1}\int_{t}^{t+\epsi}H_sf\left(\frac{K_s X^\epsi_{s}}{\epsi}\right)ds, \\
a^{2}_t &:=H_t w(K_t,L_t,M_t),
\end{align*}
where, for constants $k \geqslant 0$ and $l, m > 0$,
\begin{equation*}
w(k,l,m) :=\frac{\int_{\bR}f\left(\frac{k}{m}x\right)\exp\left(-2\frac{l}{m}G\left(x\right)\right)dx}{\int_{\bR}\exp\left(-2\frac{l}{m}G\left(x\right)\right)dx}.
\end{equation*}
Note that the choices for $\epsi$ and $\kappa\in(0,1)$ ensure that $t+\epsi\leqslant T$. Furthermore, by Lemma \ref{lem:sandwich:2}, we have
\begin{equation}
\label{eq:pf:prop.limit.integrals:inequality}
a^{\epsi, \delta, n, -}_t \leqslant a^{1,\epsi,\delta}_t \leqslant a^{\epsi, \delta, n, +}_t.
\end{equation}	
We want to show that, in probability, $\lim_{\epsi \to 0} a^{1,\epsi}_t = a^{2}_t$. To this end, we use the subsequence criterion for convergence in probability, cf.~\cite[Lemma 4.2]{kallenberg.02}. Let $(\epsi_m)_{m \in \bN}$ be a sequence of positive real numbers converging to zero. By Lemma \ref{lem:ergodic elementary} and the subsequence criterion, there exists a subsequence $(m_k)_{k \in \bN}$ of $\bN$ such that 
\begin{align*}
\lim_{k \to \infty} v^{\epsi_{m_k}, \delta, \pm}_{Y^{\epsi_{m_k}}_0}(c^{n,\pm}_t,K^{n,\pm}_t,L^{n,\mp}_t,M^{n,\mp}_t)   = w^{\delta,\pm}(K^{n,\pm}_t, L^{n,\mp}_t, M^{n,\mp}_t) \;\; \text{a.s.}
\end{align*}
Hence, 
\begin{align*}
\lim_{k \to \infty} a^{\epsi_{m_k}, \delta, n, -}_t &= \left(1-\delta\right)^2 H_t \frac{c^{n,-}_t}{c^{n,+}_t }  w^{\delta,-}(K^{n,-}_t, L^{n,+}_t, M^{n,+}_t)\;\; \text{a.s.}, \\
\lim_{k \to \infty} a^{\epsi_{m_k}, \delta, n, +}_t &= \left(1+\delta\right)(H_t + \delta) \frac{c^{n,+}_t}{c^{n,-}_t} w^{\delta,+}(K^{n,+}_t, L^{n,-}_t, M^{n,-}_t) \;\; \text{a.s.}
\end{align*}
Moreover, \eqref{eq:pf:prop.limit.integrals:inequality} and $\epsi_\delta > 0$ give
\begin{align}
\lim_{k \to \infty} a^{\epsi_{m_k}, \delta, n, -}_t &\leqslant \liminf_{k \to \infty} a^{1,\epsi_{m_k},\delta}_t = \liminf_{k \to \infty} a^{1,\epsi_{m_k}}_t  \leqslant \limsup_{k \to \infty}  a^{1,\epsi_{m_k}}_t = \limsup_{k \to \infty}  a^{1,\epsi_{m_k},\delta}_t \notag \\
&\leqslant \lim_{k \to \infty} a^{\epsi_{m_k}, \delta, n, +}_t\;\; \text{a.s.}
\label{eq:pf:prop.limit.integrals:inequality limit}
\end{align}
Finally, using that $\Upsilon^{n, \pm}_t$ (defined in \eqref{eq:proc.approx}) converges almost surely to $\Upsilon_t$ as $n \to \infty$ for $\Upsilon \in \{K, L, M\}$, we obtain by dominated convergence (and continuity of $f$) that
\begin{align*}
\lim_{\delta \to  0} \lim_{n \to \infty} \lim_{k \to \infty} a^{\epsi_{m_k}, \delta, n, -}_t = H_t w(K_t, L_t,M_t)
= \lim_{\delta \to  0} \lim_{n \to \infty} \lim_{k \to \infty} a^{\epsi_{m_k}, \delta, n, +}_t.
\end{align*}
Together with \eqref{eq:pf:prop.limit.integrals:inequality limit}, this shows that
\begin{equation*}
 \liminf_{k \to \infty} a^{1,\epsi_{m_k}}_t  = \limsup_{k \to \infty}  a^{1,\epsi_{m_k}}_t = H_t w(K_t, L_t,M_t) = a^2_t\ \text{a.s.}
\end{equation*}
The assertion in turn follows from the subsequence criterion.
\end{proof}

\subsection{Concatenation of the local estimates}
\label{section.concatenation}

We now piece together the local estimates from Proposition~\ref{proposition.limit.integrals} to establish Theorem~\ref{theorem}. For each $\epsi \in (0, T)$, define as above the product-measurable processes $(a^{1,\epsi}_t)_{t \in [0, T]}$ and $(a^{2}_t)_{t \in [0, T]}$ by\footnote{The indicator of the set $\{0\leqslant t\leqslant T -\epsi\}$ is needed now as we define $a^{1,\epsi}_t$ on $[0, T]$ for fixed $\epsi > 0$.}
\begin{align*}
a^{1,\epsi}_t &= \one_{\{0\leqslant t\leqslant T -\epsi\}}\epsi^{-1}\int_{t}^{t+\epsi}H_sf\left(\frac{K_s X^\epsi_{s}}{\epsi}\right)ds, \quad a^{2}_t = H_t w(K_t, L_t,M_t).
\end{align*}
It follows from Fubini's theorem that, for each $t \in [0, T]$:
\begin{align}
\int_0^t a^{1,\epsi}_s ds 
&=\int_{0}^T \int_{0}^T  \one_{\{0 \leqslant s \leqslant t \wedge (T -\epsi)\}} \one_{\{s \leqslant r \leqslant s + \epsi\}} \epsi^{-1} H_rf\left(\frac{K_r X^\epsi_{r}}{\epsi}\right)dr ds
\notag\\
&=  \int_{0}^T \int_{0}^T  \one_{\{r - \epsi \vee 0 \leqslant s \leqslant r\wedge  t \wedge (T -\epsi)\}} ds \one_{\{0 \leqslant r \leqslant (t+ \epsi) \wedge T\}} \epsi^{-1} H_rf\left(\frac{K_r X^\epsi_{r}}{\epsi}\right)dr \notag \\
&=  \int_{0}^T \int_{0}^T  \left(\one_{\{r - \epsi \vee 0 \leqslant s \leqslant r\}}  - \one_{\{ t \wedge (T -\epsi) < s \leqslant r\}} \right) ds \one_{\{0 \leqslant r \leqslant (t+ \epsi) \wedge T\}} \epsi^{-1} H_rf\left(\frac{K_r X^\epsi_{r}}{\epsi}\right)dr \notag \\
&= \int_0^{\epsi} \epsi^{-1}  r  H_rf\left(\frac{K_r X^\epsi_{r}}{\epsi}\right)dr+\int_{\epsi}^{(t +\epsi) \wedge T} H_rf\left(\frac{K_r X^\epsi_{r}}{\epsi}\right)dr  \notag  \\
&\quad - \int_{ t \wedge (T -\epsi)}^{(t+\epsi) \wedge T}\epsi^{-1}\Big(r -\big(t\wedge(T-\epsi\big)\Big) H_rf\left(\frac{K_r X^\epsi_{r}}{\epsi}\right)dr \label{eq:pre.concat:Fubini}\\
&\leqslant \int_{0}^T H_sf\left(\frac{K_s X^\epsi_{s}}{\epsi}\right)ds.
\label{eq:concat:Fubini}
\end{align}
Rearranging and recalling the growth condition \eqref{eq:condition.f} for $f$ and the fact that $K$ and $L$ are uniformly bounded from above by $1/\kappa$ yields
\begin{align*}
\sup_{t \in [0, T]} \left|\int_0^t a^{1,\epsi}_s ds - \int_{0}^t H_sf\left(\frac{K_s X^\epsi_{s}}{\epsi}\right)ds\right| &\leqslant  \int_0^{\epsi} H_s f\left(\frac{K_s X^\epsi_{s}}{\epsi}\right)ds \notag \\
&\quad + \sup_{t \in [0, T]} \left(\int_t^{(t+\epsi) \wedge T} H_s f\left(\frac{K_s X^\epsi_{s}}{\epsi}\right)ds \right) \notag \\
&\quad + \sup_{t \in [0, T-\epsi]} \left(\int_t^{t+\epsi} H_s f\left(\frac{K_s X^\epsi_{s}}{\epsi}\right)ds \right) \notag \\
&\leqslant 3 \sup_{t \in [0, T]} \left(\int_t^{(t+\epsi) \wedge T} \frac{1}{\kappa} C_f \left(\left(\left|\frac{X^\epsi_{s}}{\epsi}\right| \frac{1}{\kappa}\right)^{q^\prime} + 1\right) ds \right)  \notag \\
&\leqslant 3 C_f \left(\frac{1}{\kappa}\right)^{1 + q^\prime} \ \epsi \left( \sup_{t \in [0, T]}\left|\frac{X^\epsi_{t}}{\epsi}\right|^{q^\prime} + \kappa^{q^\prime} \right).
%\label{eq:concat:Sp}
\end{align*}
Thus, it follows from Lemma \ref{lem.max.inequality} that 
\begin{equation}
\lim_{\epsi \to 0} \bE \left[\sup_{t \in [0, T]} \left|\int_0^t a^{1,\epsi}_s ds - \int_{0}^t H_sf\left(\frac{K_s X^\epsi_{s}}{\epsi}\right)ds\right|^p \right] =0.
\label{eq:concat:Sp:null:sequence}
\end{equation}

Now, by Proposition \ref{proposition.limit.integrals}, for each $t \in [0, T)$, we have
\begin{equation*}
\plim_{\epsi \to 0} a^{1,\epsi}_t  = a^2_t.
\end{equation*}

Next, recall that Assumptions \ref{ass.novikov} and \ref{ass.integrability} are satisfied due to the uniform boundedness assumption~\eqref{eq:ass kappa}. Therefore, Lemma~\ref{lem:app:integrability estimate} ensures that $a^2$ is in $ L^p (\bP\otimes\mathrm{Leb}_{|[0,T]})$. Moreover, Jensen's inequality and computations similar to the ones leading to \eqref{eq:concat:Fubini} give 
\begin{align*}
\int_{0}^{T}\big|a^{1,\epsi}_t\big|^{p(1+\eta)}dt\leqslant \int_{0}^{T}H_s^{p(1+\eta)}f\left(\frac{K_tX^\epsi_{t}}{\epsi}\right)^{p(1+\eta)}dt,
\end{align*} 
which, with Lemma \ref{lem:U.I} yields
\begin{align}
&\sup_{\epsi > 0}\bE \left[\int_{0}^{T}\big|a^{1,\epsi}_t\big|^{p(1+\eta)}dt\right]  < \infty. \label{eq:concat:U.I estimate}
\end{align}
By de la Vall\'ee-Poussin's criterion for uniform integrability (cf., e.g., the remark before \cite[Lemma~4.10]{kallenberg.02}), \eqref{eq:concat:U.I estimate} implies that $(a^{1,\epsi})^p$ is uniformly integrable with respect to $\bP\otimes\text{Leb}_{|[0,T]}$.

Hence, it follows from Lemma \ref{lemma.abstract} that
\begin{equation}
\label{eq:concat:abstract}
\lim_{\epsi \to 0} \bE \left[\sup_{t \in [0, T]} \left|\int_0^t a^{1,\epsi}_s ds - \int_{0}^t a^2_s ds\right|^p\right] = 0.
\end{equation}
Theorem \ref{theorem} now follows by putting together \eqref{eq:concat:abstract} and \eqref{eq:concat:Sp:null:sequence}.
\qed

\begin{remark}
	\label{rem:comparison:Vere}
	To prove the variant of the results mentioned in Remark~\ref{rem:corollary.Vere}, the argument in Section \ref{section.concatenation} needs to be changed at two places. The first is the use of moments of $X^\epsi/\epsi$ to obtain \eqref{eq:Y0epsi}. This requires a version of Lemma~\ref{lemma.uniform.moment.bar} valid under the alternative assumptions of Remark~\ref{rem:corollary.Vere}. To this end, we argue as in Veretennikov~\cite[Lemmas 1-5]{veretennikov.97}: Fix $\epsi > 0$ and compare the SDE satisfied by $\overline{X}^\epsi_\cdot=X^\epsi_{\epsi^2 \cdot}$ under $\bQ$,
	\begin{align*}
		d\overline{X}^\epsi_s =-\sigma^2 L_{\epsi^2 s} g\Big(M_{\epsi^2 s}\overline{X}^\epsi_s\Big)ds +\sigma  d\overline{W}^{\epsi,\bQ}_s, \qquad \overline{X}^\epsi_0=x^\epsi_0/\epsi,
	\end{align*}
	where $\overline{W}^{\epsi,\bQ}$ is some  $\bQ$-Brownian motion, to the SDE satisfied by the process $\overline{X}^{V,\epsi}$ under $\bQ$ given by\footnote{This process is denoted by $v$ in \cite[Lemma 2]{veretennikov.97}. Note, moreover, that  $\overline{X}^{V,\epsi}$ depends on $\epsi$ only through the Brownian motion $\overline{W}^{\epsi,\bQ}$.}
	\begin{equation}
		d \overline{X}^{V,\epsi}_s = -\frac{\sigma^2 C}{2} \frac{1}{\overline{X}^{V,\epsi}_s}ds+\sigma d\overline{W}^{\epsi,\bQ}_s+dA_s, \qquad\overline{X}^{V,\epsi}_0=\frac{x^\epsi_0}{\epsi}\vee 2,
	\end{equation}
	where $A$ is a non-decreasing process increasing on $\{\overline{X}^{V,\epsi}=1\}$, starting at $0$ and such that
	\begin{equation*}
		A_t=\int_0^t \one_{\{\overline{X}^{V,\epsi}_s=1\}}dA_s,\qquad \mbox{and} \qquad \bE_{\bQ}\bigg[\int_{0}^{\epsi^{-2}T} \one_{\{\overline{X}^{V,\epsi}_s=1\}}ds\bigg]=0.
	\end{equation*}
	The process $\overline{X}^{V,\epsi}$ has a non-sticky reflecting boundary at $1$. In view of~\eqref{eq:condition.Veretennikov}, \cite[Lemma 2]{veretennikov.97} then shows that
	\begin{align}
	\label{eq:comparision.Veretennikov}
	\bQ\Big[\big|\overline{X}^\epsi_t\big|\leqslant \big|\overline{X}^{V,\epsi}_t\big|,~\forall~t\in[0,\epsi^{-2}T]\Big]=1.
	\end{align}
	The existence of all moments for $X^\epsi/\epsi$ then follows from the existence of all moments for $\overline{X}^{V,\epsi}$, cf.~\cite[Lemmas 3-5]{veretennikov.97} and the assumption $\limsup_{\epsi \to 0} |x^\epsi_0|/\epsi < \infty$ made on the initial condition.
	
	The second adaptation is necessary in the arguments after~\eqref{eq:concat:Fubini}. Fix $0<\epsi\leqslant 1$. Starting from \eqref{eq:pre.concat:Fubini}, using that $f$ and $H$ are nonnegative and $f$ is non-decreasing, the comparison result \eqref{eq:comparision.Veretennikov} (which remains true under $\bP$), the elementary inequality $(a+b)^p \leqslant 2^p (a^p+b^p)$ for $a,b\geqslant 0$, Jensen's inequality (using also that $p \geq  1$ and $\epsi\leqslant 1$), the growth condition \eqref{eq:condition.f} for $f$, Fubini's theorem and the moments assumptions on $H$ and $K$, we obtain for fixed $t \in [0, T]$
	\begin{align}
		&\bE\bigg[\left|\int_0^t a^{1,\epsi}_s ds - \int_{0}^t H_sf\left(\frac{K_s X^{\epsi}_{s}}{\epsi}\right)ds\right|^p\bigg] \notag \\
		&\quad \leqslant \bE\bigg[\bigg|\int_0^{\epsi} H_sf\left(K_s \overline{X}^\epsi_{\epsi^{-2} s}\right)ds+2\int_{t\wedge(T-\epsi)}^{(t+\epsi) \wedge T} H_sf\left(K_s \overline{X}^\epsi_{\epsi^{-2} s}\right)ds\bigg|^p\bigg] \notag \\
		&\quad \leqslant \bE\bigg[\bigg|\int_0^{\epsi} H_sf\left(K_s \overline{X}^{V,\epsi}_{\epsi^{-2} s}\right)ds+2\int_{t\wedge(T-\epsi)}^{(t+\epsi) \wedge T} H_sf\left(K_s \overline{X}^{V,\epsi}_{\epsi^{-2} s}\right)ds\bigg|^p\bigg] \notag \\
		&\quad \leqslant  2^p\bE\bigg[\int_0^{\epsi} H^p_s f\left(K_s  \overline{X}^{V,\epsi}_{\epsi^{-2}s}\right)^p ds\bigg]+ 4^p\bE\bigg[\int_{t\wedge(T-\epsi)}^{(t+\epsi) \wedge T} H_s^p f\left(K_s  \overline{X}^{V,\epsi}_{\epsi^{-2}s}\right)^p ds\bigg] \notag\\
		&\quad \leqslant  4^pC_f^p\int_0^{\epsi} \bE\bigg[H^p_s \Big(K^{pq'}_s  \Big|\overline{X}^{V,\epsi}_{\epsi^{-2}s}\Big|^{pq'}+1\Big)\bigg]ds \notag \\
		&\quad\quad+ 8^pC_f^p\int_{t\wedge(T-\epsi)}^{(t+\epsi) \wedge T} \bE\bigg[H^p_s \Big(K^{pq'}_s  \Big|\overline{X}^{V,\epsi}_{\epsi^{-2}s}\Big|^{pq'}+1\Big)\bigg]ds.
		\label{eq:rem:Veretennikov estimate}
	\end{align}
	Next, apply the Cauchy--Schwarz inequality twice and use the integrability Assumption~\ref{ass.novikov} and \cite[Lemma 5]{veretennikov.97}. This together with the elementary inequality $x^{1/2} \leq x + 1$ for $x \geq 0$ gives for some constant $C$ independent of $s$,
	\begin{align*}
	\bE\bigg[H^p_s \Big(K^{pq'}_s  \Big|\overline{X}^{V,\epsi}_{\epsi^{-2}s}\Big|^{pq'}+1\Big)\bigg] &\leqslant\bE\bigg[H^{2p}_s K^{2pq'}_s\bigg]^{\frac{1}{2}}\bE_\bQ\bigg[\frac{d\bP}{d\bQ}  \Big|X^{V,\epsi}_{\epsi^{-2}s}\Big|^{2pq'}\bigg]^{\frac{1}{2}} +\bE\big[H^p_s\big]  \\
		&\leqslant\bE\bigg[H^{2p}_s K^{2pq'}_s\bigg]^{\frac{1}{2}}\bE\bigg[\left(\frac{d\bP}{d\bQ}\right) \bigg]^{\frac{1}{4}}\bE_\bQ\bigg[ \Big|X^{V,\epsi}_{\epsi^{-2}s}\Big|^{4pq'}\bigg]^{\frac{1}{4}} +\bE\big[H^p_s\big] \notag \\
		&\leqslant C \bE\bigg[H^{2p}_s K^{2pq'}_s\bigg]^{\frac{1}{2}} + \bE\big[H^p_s\big] 	\leqslant C \left(\bE\bigg[H^{2p}_s K^{2pq'}_s\bigg] + 1 \right) + \bE \big[H^p_s\big].
	\end{align*}
	Plugging this into \eqref{eq:rem:Veretennikov estimate}, using Fubini's theorem and Assumption~\ref{ass.integrability} finally gives
	\begin{align*}
		\bE\bigg[&\left|\int_0^t a^{1,\epsi}_s ds - \int_{0}^t H_sf\left(\frac{K_s X^{\epsi}_{s}}{\epsi}\right)ds\right|^p\bigg] \to 0 ~~\mbox{as}~~\epsi\to 0.
	\end{align*}
	This establishes convergence in $L^p$ (and a fortiori in probability) for all $t\in[0,T]$.\footnote{Uniform convergence results would require extending our maximal inequality to sublinear mean-reversion speeds.}
\end{remark}

\appendix
\section{Integrability Results}
\label{app:UI}

In this appendix, we establish moment estimates that are needed at various stages of the proof of Theorem \ref{theorem}. Most importantly, we show in Lemma \ref{lem:U.I} that the $p$-th moment of the expression on the left-hand side of \eqref{eq:thm:up} is uniformly integrable. This is crucial both for the reduction to bounded coefficients in Section~\ref{sec:reduction} and the concatenation argument in Section~\ref{section.concatenation}. The main ingredient to this result is to establish that $X^\epsi/\epsi$ has uniformly bounded moments. This is done in Lemma~\ref{lemma.uniform.moment.bar}.

\subsection{A First Moment Estimate}
The following result is the key ingredient for Lemma \ref{lem:U.I}. It is also used in the proof of Lemma~\ref{lem:ergodic elementary}.
\begin{lemma}
	\label{lemma.uniform.moment.bar}
Let $N \geqslant 2$ be a real number. Suppose that Assumption~\ref{ass.novikov} is satisfied and 
	\begin{align*}
	\label{eq:lemma.uniform.moment.bar:assumption}
	&\bE\left[\sup_{u \in [0, T]} c_u^{2N} \right] +\bE\left[\sup_{u\in[0,T]}\left(L_{u}c_{u}\wedge M_{u}\right)^{-\frac{2(q+1)N}{q-1}}\right]<\infty \qquad \mbox{if}~~q>1,
	\end{align*}
	or
	\begin{align*}
  \bE\left[\sup_{u \in [0, T]} c_u^{2N} \right]<\infty \quad \text{and}\quad  \mathrm{essinf}_{u\in[0,T]}(L_{u}c_{u}\wedge M_{u})>0 \qquad \mbox{if}~~ q=1.
	\end{align*}
Then for each $0\leqslant n\leqslant N$, there exists a constant $\overline{C}_n > 0$ such that for all $t \in [0, T]$ and $\epsi > 0$,
	\begin{equation}
	\label{eq:lemma.uniform.moment.bar:Q}
	\bE_{\bQ}\left[\left|\frac{X^\epsi_{t}}{\epsi}\right|^n\right]\leqslant \overline{C}_n,
	\end{equation}
where $\bQ$ is the probability measure defined in \eqref{eq:proba.change.drift.Delta}.
\end{lemma}
\begin{proof}
Note that $\frac{d \bQ}{d \bP}$ is square integrable under $\bP$ by Assumption~\ref{ass.novikov} and Novikov's criterion. The Cauchy-Schwarz inequality implies that
	\begin{equation*}
		\bE_{\bQ}\left[\left|A\right| \right] \leqslant \bE\left[A^2\right]^\frac{1}{2}  \bE\left[\left(\frac{d \bQ}{d \bP}\right)^2 \right]^{\frac{1}{2}} < \infty
	\end{equation*}
	for any random variable $A$ that is square-integrable under $\bP$. This estimate will be used throughout the proof without further mention.
	
Define the process $(\overline{X}^\epsi_s)_{s \in [0, \epsi^{-2} T]}$ by
	\begin{equation*}
	\overline{X}^\epsi_s := X^\epsi_{\epsi^2 s}/\epsi.
	\end{equation*}
	Then, $\overline{X}^\epsi$ satisfies under $\bQ$ the SDE
	\begin{equation}
	\label{eq:pf:lemma.uniform.moment.bar:SDE Xbar}
	d\overline{X}^\epsi_s=-L_{\epsi^2s}c_{\epsi^2s}g\left(M_{\epsi^2s}\overline{X}^\epsi_s\right)ds+\sqrt{c_{\epsi^2s}}d\overline{W}^{\epsi,\bQ}_s, \quad \overline{X}^\epsi_0 = x^\epsi_0/\epsi.
	\end{equation}
	Moreover, for $2\leqslant n\leqslant N$, the $n$-th power of $\left|\overline{X}^\epsi\right|$ satisfies the SDE
	\begin{align}
	d\left|\overline{X}^\epsi_s\right|^n =& \left(-n\ \overline{X}^\epsi_s \left|\overline{X}^\epsi_s\right|^{n-2}L_{\epsi^2s}c_{\epsi^2s}g\left(M_{\epsi^2s}\overline{X}^\epsi_s\right)+\frac{n(n-1)}{2}c_{\epsi^2s}\left|\overline{X}^\epsi_s\right|^{n-2} \right) ds \notag \\
	&+n\ \sign\left(\overline{X}^\epsi_s\right)\left|\overline{X}^\epsi_s\right|^{n-1}\sqrt{c_{\epsi^2s}}d\overline{W}^{\epsi,\bQ}_s, \quad \left|\overline{X}^\epsi_0\right|^n = \left(x^\epsi_0/\epsi\right)^n,
	\label{eq:pf:lemma.uniform.moment.bar:SDE Xbarn}
	\end{align}
	where $\overline{W}^{\epsi,\bQ}$ is a $\bQ$-Brownian motion. 
	
	Comparing \eqref{eq:pf:lemma.uniform.moment.bar:SDE Xbarn} for $n=2$ to the SDE
	\begin{equation*}
	d\overline{Y}^\epsi_s=c_{\epsi^2s} ds +2 \sqrt{\overline{Y}^\epsi_s}\sqrt{c_{\epsi^2s}}d\overline{B}^{\epsi,\bQ}_s, \quad \overline{Y}^\epsi_0 =  (x^\epsi_0)^2/\epsi^2,
	\end{equation*}
	where $\overline{B}^{\epsi,\bQ}_s=\int_{0}^{s}\text{sgn}\left(\overline{X}^\epsi_u\right)d\overline{W}^{\epsi,\bQ}_u$ for $s\in[0,\epsi^{-2}T]$, we obtain by Lemma \ref{lem:app:comparison} that 
	\begin{equation*}
	\bQ\left[\left|\overline{X}^\epsi_s\right|^2 \leqslant \overline{Y}^\epsi_s, \text{ for all } s \in [0, \epsi^{-2} T]\right] = 1.
	\end{equation*}
By Lemma~\ref{lem:app:existence:SDE:2} and the integrability assumption on $c$ this implies that $\sup_{t\in [0, \epsi^2T]}|\overline{X}^\epsi_t|$ has moments of orders up to $2N$ under $\bQ$. Together with H\"older's inequality, it follows that, for each $2\leqslant n\leqslant N$,
	\begin{equation*}
	\bE_{\bQ}\left[\int_0^{\epsi^{-2}T} n^2 \left|\overline {X}^\epsi_s\right|^{2(n-1)} c_{\epsi^2s} ds  \right] \leqslant n^2 \epsi^{-2} T   \bE_{\bQ}\left[\sup_{s \in [0, \epsi^{-2}T]} \left(\overline{Y}^\epsi_s\right)^{n}  \right]^{\frac{n-1}{n}}\bE_{\bQ}\left[\sup_{u \in [0, T]}c_u^{n} \right]^{\frac{1}{n}} < \infty.
	\end{equation*}
Hence, the local martingale term in \eqref{eq:pf:lemma.uniform.moment.bar:SDE Xbarn} is a $\bQ$-martingale.
	
We now show that the supremum of the positive part of the integrand in the $ds$-term in \eqref{eq:pf:lemma.uniform.moment.bar:SDE Xbarn} is $\bQ$-integrable. To this end, we compute
\begin{align*}
&\sup_{s \in [0, \epsi^{-2} T]} \left\{\left(-n\ \overline{X}^\epsi_s \left|\overline{X}^\epsi_s\right|^{n-2}L_{\epsi^2s}c_{\epsi^2s}g\left(M_{\epsi^2s}\overline{X}^\epsi_s\right)+\frac{n(n-1)}{2}c_{\epsi^2s}\left|\overline{X}^\epsi_s\right|^{n-2} \right)^+\right\}\\
&\qquad\qquad\qquad\qquad\qquad\qquad\qquad\qquad\qquad\qquad\qquad\qquad \leqslant  \sup_{s \in [0, \epsi^{-2} T]} \frac{n(n-1)}{2}c_{\epsi^2s}\left|\overline{X}^\epsi_s\right|^{n-2}.
\end{align*}
Now the claim follows from the Cauchy-Schwarz inequality, the stated integrability assumptions and the fact that $\sup_{t\in [0, \epsi^{-2}T]}|\overline{X}^\epsi_t|$ has moments of orders up to $2N$ under $\bQ$.

Thus, the function $j_{n,\epsi} : [0, \epsi^{-2} T] \to \mathbb{R} \cup \{-\infty\}$ given by
	\begin{equation}
\label{eq:ODI}
j_{n,\epsi}(t)=-n\bE_{\bQ}\left[\overline{X}^\epsi_t \left|\overline{X}^\epsi_t\right|^{n-2}L_{\epsi^2t}c_{\epsi^2t}g\left(M_{\epsi^2t}\overline{X}^\epsi_t\right)\right]+\frac{n(n-1)}{2}\bE_{\bQ}\left[c_{\epsi^2t}\left|\overline{X}^\epsi_t\right|^{n-2}\right]
\end{equation}
is well defined and bounded from above. Define the function $J_{n,\epsi}: [0, \epsi^{-2} T] \to \mathbb{R}$ by $ J_{n,\epsi}(t)=\bE_{\bQ}\left[\left|\overline{X}^\epsi_t\right|^n\right]$. After taking $\bQ$-expectations in \eqref{eq:pf:lemma.uniform.moment.bar:SDE Xbarn}, Fubini's theorem and the fundamental theorem of calculus for Lebesgue-measurable functions show that $J_{n,\epsi}$ is absolutely continuous and almost everywhere differentiable with derivative $j_{n,\epsi}$. In particular $j_{n,\epsi}$ is almost everywhere finite and Lebesgue-integrable.

We proceed to derive a differential inequality for $J_{n,\epsi}$. To this end, we first use the fact that $g$ is odd and that by condition~\eqref{eq:growth.g},  $x g(x) \geqslant 0$ and $xg(x) \geqslant a \left|x\right|^{q+1} - \widetilde a$ for some constants $a, \widetilde a > 0$ and all $x \in \mathbb{R}$. This yields
	\begin{align}
	&j_{n,\epsi}(t) \leqslant -n\bE_{\bQ}\left[\frac{L_{\epsi^2t}c_{\epsi^2t}\wedge M_{\epsi^2 t}}{M_{\epsi^2 t}}\left|\overline{X}^\epsi_t\right|^{n-2}\left(a\left|\overline{X}^\epsi_t\right|^{q+1}M^{q+1}_{\epsi^2 t}-\widetilde{a}\right)\right]+\frac{n(n-1)}{2}\bE_{\bQ}\left[c_{\epsi^2t}\left|\overline{X}^\epsi_t\right|^{n-2}\right]\notag \\
	&\quad\leqslant -an\bE_{\bQ}\left[\left(L_{\epsi^2t}c_{\epsi^2t}\wedge M_{\epsi^2 t}\right)^{q+1}\left|\overline{X}^\epsi_t\right|^{n-1+q}\right]+\bE_{\bQ}\left[\left(n \widetilde a + \frac{n(n-1)}{2}c_{\epsi^2t}\right)\left|\overline{X}^\epsi_t\right|^{n-2}\right]. 
	\label{eq:pf:lemma.uniform.moment.bar:h prime estimate} 
	\end{align}
	We now estimate the first term on the right-hand side of~\eqref{eq:pf:lemma.uniform.moment.bar:h prime estimate}. We consider two cases separately for the growth condition~\eqref{eq:growth.g}: strictly superlinear growth ($\liminf_{x \to \infty} \frac{g(x)}{x^q} >0$ for $q>1$) and linear growth ($\liminf_{x \to \infty} \frac{g(x)}{x} >0$). In the first case, we use the reverse H\"older inequality (with power $\frac{n}{n+q-1}$) and the inequality $\left|x\right|^{\frac{n+q-1}{n}}\geqslant \left|x\right|-1$ for $x\in\bR$. This gives
	\begin{align*}
	\bE_{\bQ}\bigg[\left(L_{\epsi^2t}c_{\epsi^2t}\wedge M_{\epsi^2 t}\right)^{q+1}&\left|\overline{X}^\epsi_t\right|^{n-1+q}\bigg] \geqslant \bE_{\bQ}\left[\left|\overline{X}^\epsi_t\right|^{n}\right]^{\frac{n+q-1}{n}} \bE_{\bQ}\left[\left(L_{\epsi^2t}c_{\epsi^2t}\wedge M_{\epsi^2 t}\right)^{-\frac{(q+1)n}{q-1}}\right]^{-\frac{q-1}{n}} \\
	&\qquad\geqslant\left(\bE_{\bQ}\left[\left|\overline{X}^\epsi_t\right|^{n}\right]-1\right)\bE_{\bQ}\left[\sup_{u\in[0,T]}\left(L_{u}c_{u}\wedge M_{u}\right)^{-\frac{(q+1)n}{q-1}}\right]^{-\frac{q-1}{n}}.
	\end{align*}
	In the second case, we obtain similarly from Assumption~\ref{ass.integrability} (see Remark~\ref{rem:integrability}) that
	\begin{align*}
	\bE_{\bQ}\Big[\left(L_{\epsi^2t}c_{\epsi^2t}\wedge M_{\epsi^2 t}\right)^{2}&\left|\overline{X}^\epsi_t\right|^{n}\Big] \geqslant \text{ ess\ inf} _{u\in[0,T]}\left\{(L_{u}c_{u}\wedge M_{u})^2\right\}\bE_{\bQ}\Big[\left|\overline{X}^\epsi_t\right|^{n}\Big] 
	\end{align*}
Plugging this into \eqref{eq:pf:lemma.uniform.moment.bar:h prime estimate} and setting 
$$
C_n := 
\begin{cases} 
a n\ \bE_{\bQ}\left[\sup_{u\in[0,T]}\left(L_{u}c_{u}\wedge M_{u}\right)^{-\frac{(q+1)n}{q-1}}\right]^{-\frac{q-1}{n}}, &\mbox{if $q>1$,}\\  
a n\ \mathrm{ ess\ inf} _{u\in[0,T]}\left\{(L_{u}c_{u}\wedge M_{u})^2\right\}, &\mbox{if $q=1$,}
\end{cases}
$$
this yields
	\begin{align}
	\label{eq:hn.epsi}
	&j_{n,\epsi}(t) \leqslant - C_n J_{n,\epsi}(t) + C_n\one_{\{q>1\}} + \bE_{\bQ}\left[\left(n \widetilde a + \frac{n(n-1)}{2}c_{\epsi^2t}\right)\left|\overline{X}^\epsi_t\right|^{n-2}\right]. 
	\end{align}
We now establish \eqref{eq:lemma.uniform.moment.bar:Q} for $n=2$ and use it to prove the result for general integer $n\geqslant 3$ by induction. For $n = 2$, the $\bQ$-expectation on the right hand side of \eqref{eq:hn.epsi} can be bounded by $\bE_{\bQ}\left[2 \widetilde a + \sup_{u \in [0, T]} c_{u}\right]$ and  then \eqref{eq:lemma.uniform.moment.bar:Q} follows from Lemma \ref{lem:ODE.bound}. For $n\geqslant 3$, assume that $J_{n-1,\epsi}(t)\leqslant \overline{C}_{n-1}$ for some $\overline{C}_{n-1}>0$. Then H\"older's inequality and the inequality $(a +b)^{n-1} \leq 2^{n-1}(a^{n-1}+b^{n-1})$ for $a, b \geq 0$ give
	\begin{align*}
		&j_{n,\epsi}(t) \leqslant - C_n J_{n,\epsi}(t) + C_n\one_{\{q>1\}} + 2 \left((n\tilde a)^{n-1} +\Big(\frac{n(n-1)}{2}\Big)^{n-1}\bE_{\bQ}\left[\sup_{u \in [0, T]} c_{u}^{n-1}\right]\right)^{\frac{1}{n-1}} \overline{C}^{\frac{n-2}{n-1}}_{n-1}.
	\end{align*}
Now \eqref{eq:lemma.uniform.moment.bar:Q} follows from Lemma~\ref{lem:ODE.bound}. 
\end{proof}

The following differential inequality is used in the proof of Lemma \ref{lemma.uniform.moment.bar}.
\begin{lemma}
	\label{lem:ODE.bound}
	Let $C,C' >0 \in\bR$ be constants and $J: [0, T] \to \mathbb{R}$ be an absolutely continuous function with almost everywhere derivative $j$ satisfying the differential inequality
	\begin{equation*}
		j(t)\leqslant C-C'J(t)\; \text{a.e}., \quad J(0) = J_0.
	\end{equation*}
	Then there exists a positive constant $C''$ such that $J(t)\leqslant C''$ on $[0,T]$.
\end{lemma}

\begin{proof}
	Define the function $t\mapsto C_1(t):=j(t)+C'J(t)$. Then $C_1\leqslant C$ a.e. and $J$ satisfies the ODE
	\begin{equation*}
		j(t)= C_1(t) -C'J(t)\; \text{a.e}, \quad J(0) = J_0.
	\end{equation*}
	This ODE has the explicit solution $J(t)={\rm e}^{-C't}\left(J_0+\int_{0}^{t}C_1(s){\rm e}^{C's}ds\right)$, which yields the desired bound:
	\begin{equation*}
		J(t)\leqslant J_0+\frac{C}{C'},\quad t\in[0,T]. \qedhere
	\end{equation*}
\end{proof}
Note that we could not use Gronwall inequality here, as the factor in front of the function $J$ is negative.

\subsection{A Uniform Integrability Result}
With the help of Lemma~\ref{lemma.uniform.moment.bar}, we can now establish our result on uniform integrability:

\begin{lemma}
\label{lem:U.I}
Suppose that Assumptions \ref{ass.novikov} and \ref{ass.integrability} are satisfied. Then:
\begin{align}
\label{eq:lem:U.I}
\sup_{\epsi > 0}\ \bE\Bigg[\int_{0}^{T}H^{p(1+\eta)}_t f\left(\frac{K_t X^\epsi_{t}}{\epsi}\right)^{p(1+\eta)} dt\Bigg] < \infty.
\end{align}

\end{lemma}

\begin{proof}
Using the polynomial bound \eqref{eq:condition.f} for $f$ and H\"older's inequality, we obtain
\begin{align}
\bE\Bigg[&\int_{0}^{T}H^{p(1+\eta)}_tf\left(\frac{K_tX^\epsi_t}{\epsi}\right)^{p(1+\eta)}dt\Bigg] \notag \\
&\leqslant C_f\bE\Bigg[\int_{0}^{T}H^{p(1+\eta)}_t\left|\frac{K_tX^\epsi_t}{\epsi}\right|^{pq'(1+\eta)}dt\Bigg] +C_f\bE\Bigg[\int_{0}^{T}H^{p(1+\eta)}_tdt\Bigg] \notag \\
&\leqslant C_f\bE\Bigg[\int_{0}^{T}\left(H_tK_t^{q'}\right)^{2p(1+\eta)}dt\Bigg]^{\frac{1}{2}}\bE\Bigg[\int_{0}^{T}\left|\frac{X^\epsi_t}{\epsi}\right|^{2pq'(1+\eta)}dt\Bigg]^{\frac{1}{2}} +C_f\bE\Bigg[\int_{0}^{T}H^{p(1+\eta)}_tdt\Bigg].
\label{eq:pf:lem:U.I:Holder}
\end{align}
By Assumption \ref{ass.integrability}, the first expectation in the first term and the last term on the right-hand side of \eqref{eq:pf:lem:U.I:Holder} are finite and independent of $\epsi$. By Fubini's theorem, Lemma~\ref{lemma.uniform.moment.bar}, H\"older's inequality and Assumption~\ref{ass.novikov} the second expectation in the first term is also finite,
\begin{align*}
\bE\Bigg[\int_{0}^{T}\left|\frac{X^\epsi_t}{\epsi}\right|^{2pq'(1+\eta)}dt\Bigg] &= \int_{0}^{ T}\bE\left[\left|\frac{X^\epsi_{t}}{\epsi}\right|^{2pq'(1+\eta)}\right]dt= \int_{0}^{T}\bE_{\bQ}\bigg[\frac{d\bP}{d\bQ}\left|\frac{X^\epsi_{t}}{\epsi}\right|^{2pq'(1+\eta)}\bigg]dt\\
&\leqslant\int_{0}^{T}\bE_{\bQ}\bigg[\bigg(\frac{d\bP}{d\bQ}\bigg)^2\bigg]^{\frac{1}{2}}\bE_{\bQ}\bigg[\left|\frac{X^\epsi_{t}}{\epsi}\right|^{4pq'(1+\eta)}\bigg]^{\frac{1}{2}}dt\\
&=T \bE\bigg[\frac{d\bP}{d\bQ}\bigg]^{\frac{1}{2}} \sqrt{\overline{C}_{4pq'(1+\eta)}}<\infty.
\end{align*}
This concludes the proof.
\end{proof}

\section{A Maximal Inequality for Square-Root Processes}
\label{app:maximal inequality}
In this section, we establish a maximal inequality for square-root processes which is inspired by a result of Peskir \cite{peskir.01}.\footnote{For square-root processes, the arguments from \cite{peskir.01} can be sharpened to obtain constants explicit in the model parameters. This is needed for the application of the estimate in the companion paper \cite{caye.al.17}.}  This estimate is crucial for establishing Lemma \ref{lem.max.inequality} which in turn is necessary to concatenate the infinitesimal estimates from Proposition~\ref{proposition.limit.integrals} to establish Theorem~\ref{theorem} in Section~\ref{section.concatenation}, but also of independent interest. 

\begin{proposition}
\label{prop:CIR process}
Let $(Y_t)_{t \geqslant 0}$ be the unique strong solution of the SDE
\begin{equation}
\label{eq:prop:CIR process:CIR SDE}
d Y_t = \nu (\theta -  Y_t) \, dt + \sigma \sqrt{Y_t} dB_t, \quad Y_0 = y_0,
\end{equation}
where $y_0 > 0$, the constants $\nu, \theta, \sigma > 0$ satisfy $\frac{2\nu\theta}{\sigma^2}<1$, and $(B_t)_{t\geqslant 0}$ is a Brownian motion on some filtered probability space. Set $\gamma := \frac{2 \nu}{\sigma^2}$ and, for $n \in \mathbb{N}$, define
\begin{align*}
C^1(y_0,\gamma,\sigma,n) &:= \left(1+ 8^n \left(y_0^n+ \gamma^{-2n} + 4^n \sigma^n\right) \right)\left(2 + \left(2^n y_0^{n-1} + \frac{2 + 8^n \gamma^{-2n}}{y_0} \right) \left(1 + \frac{12}{y_0} \left((n+1)!\right) \gamma^{-1} \right)\right), \\ 
C^2(y_0, \gamma,n) &:= 4^n \gamma^{-n}\left(2 + \left(2^n y_0^{n-1} + \frac{2 + 8^n \gamma^{-2n}}{y_0} \right) \left(1 + \frac{12}{y_0} \left((n+1)!\right) \gamma^{-1} \right)\right).
\end{align*}
 Then for any $n \geqslant 0$ and any finite stopping time $\tau$,
\begin{align}
\label{pre:CIR process}
\bE\left[\max_{0\leqslant t \leqslant \tau} Y_t^{n}\right]  \leqslant& C^1(y_0,\gamma,\sigma,n) + C^2(y_0,\gamma,n) \bE\left[\log(\tau\vee 1)^n \right].
\end{align}
\end{proposition} 

\begin{proof}
For $n = 0$, the claim is trivial. So fix $n > 0$. The SDE satisfied by $Y^n$ for $n\in\bN\backslash\{0\}$ is
\begin{align*}
d\left(Y_t^n\right) =\left(n\left(\nu\theta+\frac{n-1}{2}\sigma^2\right)\left(Y_t^n\right)^{\frac{n-1}{n}}-n\nu\left(Y_t^n\right)\right)dt+n\sigma\left(Y_t^n\right)^{\frac{n-\frac{1}{2}}{n}}dB_t, \quad Y^n_0=y^n_0.
\end{align*}
We sharpen the arguments of \cite{peskir.01} in the present context. We first compute for $z\geqslant y^n_0$ the scale function $S$ and the speed measure $m$ of the diffusion $Y^n$. The derivative of $S$ is given by
\begin{align*}
S^\prime(z) &= \exp\left(-2\int_{y^n_0}^{z}\frac{n\left(\nu\theta+\frac{n-1}{2}\sigma^2\right)y^{\frac{n-1}{n}}-n\nu y}{n^2\sigma^2y^{\frac{2n-1}{n}}}dy\right)=  \bigg(\frac{z}{y^n_0}\bigg)^{-\frac{2\left(\nu\theta+\frac{n-1}{2}\sigma^2\right)}{n\sigma^2}}\exp\left(\frac{2\nu}{\sigma^2}\Big(z^{\frac{1}{n}}-y_0\Big)\right).
\end{align*}
The speed measure $m$ in turn is given by
\begin{align*}
m(y^n_0,z] &= \int_{y^n_0}^{z}\frac{2}{n^2\sigma^2S'(y)}y^{-\frac{2n - 1}{n}}dy\\
&= \frac{2}{n^2\sigma^2}  \exp\left(\frac{2\nu y_0}{\sigma^2}\right)y^{-\frac{2\nu\theta}{\sigma^2}+1-n}_0 \int_{y^n_0}^{z}y^{\frac{-(2n -1)\sigma^2 + 2\left(\nu\theta+\frac{n-1}{2}\sigma^2\right)}{n\sigma^2}}\exp\left(-\frac{2\nu}{\sigma^2}y^{\frac{1}{n}}\right)dy\\
&=\frac{2}{n^2\sigma^2}  \exp\left(\frac{2\nu y_0}{\sigma^2}\right)y^{-\frac{2\nu\theta}{\sigma^2}+1-n}_0\int_{y^n_0}^{z}y^{\frac{2 \nu \theta }{n\sigma^2}-1}\exp\left(-\frac{2\nu}{\sigma^2}y^{\frac{1}{n}}\right)dy\\
&=\frac{2}{n^2\sigma^2}  \exp\left(\frac{2\nu y_0}{\sigma^2}\right)y^{-\frac{2\nu\theta}{\sigma^2}+1-n}_0\int_{y_0}^{z^{\frac{1}{n}}}x^{\frac{2\nu\theta}{\sigma^2}-1}\exp\left(-\frac{2\nu}{\sigma^2}x\right)dx.
\end{align*}
Define the function $F: [y^n_0, \infty) \to [0, \infty)$ (as in \cite[Theorem 2.5]{peskir.01}) for $x \geqslant y^n_0$ by
\begin{align}
F(x) = \int_{y^n_0}^{x}m(y^n_0,z]S'(z)dz &= \frac{2}{n\sigma^2}\int_{y^n_0}^{x}z^{-\frac{2\left(\nu\theta+\frac{n-1}{2}\sigma^2\right)}{n\sigma^2}}\exp\left(\frac{2\nu}{\sigma^2}z^{\frac{1}{n}}\right)\int_{y_0}^{z^{\frac{1}{n}}}y^{\frac{2\nu\theta}{\sigma^2}-1}\exp\left(-\frac{2\nu}{\sigma^2}y\right)dy\ dz \notag \\
&= \frac{2}{\sigma^2}\int_{y_0}^{x^{\frac{1}{n}}} v^{-\frac{2\nu\theta}{\sigma^2}} {\rm e}^{\frac{2\nu}{\sigma^2} v}\int_{y_0}^{v}y^{\frac{2\nu\theta}{\sigma^2}-1}{\rm e}^{-\frac{2\nu}{\sigma^2} y} dy\ dv \notag\\
&= \frac{2}{\sigma^2}\int_{y_0}^{x^{\frac{1}{n}}} v^{-\gamma \theta} {\rm e}^{\gamma v}\int_{y_0}^{v}y^{\gamma \theta -1}{\rm e}^{-\gamma y} dy\ dv,
\label{eq:Peskir:F}
\end{align}
where $\gamma := \frac{2 \nu}{\sigma^2}$. It is not difficult to check that $F$ is strictly increasing with $\lim_{x \to \infty} F(x) = +\infty$, continuously differentiable on $[y^n_0, \infty)$ and twice continuously differentiable on $(y^n_0, \infty)$.
Next, define the function $G: [0, \infty) \to [0, \infty)$ by $G(x)=\one_{\{x\geqslant y^n_0\}}F(x)$. Note that $G$ is twice continuously differentiable everywhere, except at $y^n_0$. Applying It\^o's formula as in \cite[Exercise 4.20]{shreve.04} to $G(Y_t^n)$ (note that the infinitesimal generator of the diffusion $Y^n$ applied to $G$ gives $0$ on $[0,y^n_0)$ and $1$ on $(y^n_0,+\infty)$, and that $G(Y^n_0)=0$), we obtain
\begin{equation*}
G(Y_t^n)\leqslant t+\int_{0}^{t} G^\prime\left(Y^n_t\right)n \sigma \left(Y_t^n\right)^{\frac{n-\frac{1}{2}}{n}} dB_s.
\end{equation*}
By localization of $Y^n$ and the monotone convergence theorem ($G$ is nondecreasing), it follows that, for any finite stopping time $\tau$,
\begin{equation*}
\bE\left[G(Y_\tau^n)\right] \leqslant  \bE[\tau].
\end{equation*}
Denote by $H: [0, \infty) \to [y^n_0, \infty)$ the inverse of $F$, which like $F$ is increasing and continuously differentiable on $[0, \infty)$. Then by Lenglart's domination principle in the form of \cite[Lemma 2.1]{peskir.01} with $Z_t=G(Y^n_t)$ and $A_t=t$, we obtain for any finite stopping time $\tau$,\footnote{Note that the assumption that $H(0)=0$ can be replaced by $H(0) \geqslant 0$ in \cite[Lemma 2.1]{peskir.01}.}
\begin{equation*}
\bE\left[\sup_{0\leqslant t\leqslant\tau} H(G(Y_t^n))\right]\leqslant\bE\left[\widetilde{H}\left(\tau\right)\right],
\end{equation*}
where the function $\widetilde{H}: [0, \infty) \to [0, \infty]$ is given by
\begin{equation*}
\widetilde{H}(y)=y\int_{y}^{\infty}\frac{1}{z}H^\prime(z)dz+2H(y).
\end{equation*}
Now using that $H(G(x)) = x$ for $x \in [y^n_0, \infty)$ and $H(G(x)) = H(0) = y^n_0 \geqslant x$ for $x \in [0, y^n_0)$, we obtain for any finite stopping time $\tau$,
\begin{align}
\label{eq:Peskir.ineq}
\bE\left[\sup_{0\leqslant t\leqslant\tau} Y_t^n\right]\leqslant\bE\left[\sup_{0\leqslant t\leqslant\tau}H(G(Y_t^n))\right]\leqslant \bE\left[\widetilde{H}\left(\tau\right)\right].
\end{align}
We proceed to estimate $\widetilde H(y)$. A change of variable yields
\begin{equation}
\label{eq:Peskir:tilde H estimate}
\widetilde{H}(y) \leqslant \left(\sup_{x\geqslant y^n_0}\frac{F(x)}{x}\int_{x}^{\infty}\frac{dz}{F(z)}+2\right) H(y), \quad y \in [0, \infty).
\end{equation}
In order to estimate both factors on the right-hand side of \eqref{eq:Peskir:tilde H estimate}, we need to establish lower and upper bounds for the function $F$.

First, we establish an upper bound for $F$. It follows from \eqref{eq:Peskir:F} and the assumption $\gamma \theta <1$ that for $x\in [y^n_0, \infty)$, 
\begin{align}
F(x) &= \frac{2}{\sigma^2}\int_{y_0}^{x^{\frac{1}{n}}} v^{-\gamma \theta} {\rm e}^{\gamma v}\int_{y_0}^{v}y^{\gamma \theta -1}{\rm e}^{-\gamma y} dy\ dv \leqslant  \frac{2}{\sigma^2} y_0^{-\gamma \theta}  y_0^{\gamma \theta-1} \int_{y_0}^{x^{\frac{1}{n}}} {\rm e}^{\gamma v}\int_{y_0}^{v}{\rm e}^{-\gamma y} dy\ dv  \notag \\
&\leqslant \frac{2}{\gamma \sigma^2 y_0}\int_{y_0}^{x^{\frac{1}{n}}}{\rm e}^{\gamma v}\left({\rm e}^{-\gamma y_0}-{\rm e}^{-\gamma v}\right)dv
\leqslant \frac{2 {\rm e}^{-\gamma y_0}}{\gamma \sigma^2 y_0}\int_{y_0}^{x^{\frac{1}{n}}}{\rm e}^{\gamma v}dv \leqslant  \frac{2 {\rm e}^{-\gamma y_0}}{\gamma^2 \sigma^2 y_0}\left({\rm e}^{\gamma x^{\frac{1}{n}}}-{\rm e}^{\gamma y_0}\right) \notag \\
&\leqslant \frac{2 {\rm e}^{-\gamma y_0}}{\gamma^2 \sigma^2 y_0}{\rm e}^{\gamma x^{\frac{1}{n}}}. \label{eq.sup.dom.ineq1}
\end{align}
Next, we establish a lower bound for $F$. Set
\begin{equation*}
\bar{x} := 2^n y_0^n + 1 + 8^n \gamma^{-2n} \geqslant 2^n y_0^n + \left(\frac{4 \log(2)}{\gamma} \vee \frac{8}{\gamma^2}\right)^n \geqslant \left(y_0 + \frac{2 \log(2)}{\gamma} \right)^n\vee \frac{8^n}{\gamma^{2n}},
\end{equation*}
where the first inequality follows from the fact that $\frac{8}{\gamma^2} > \frac{4 \log(2)}{\gamma}$ for $\gamma < 2/\log(2)$ and $\frac{4 \log(2)}{\gamma} \leq 2 \log(2)^2 < 1$ for $\gamma \geq 2/\log(2)$, and the second inequality follows from the elementary inequality $(a + b)^n \leqslant 2^n (a^n+ b^n)$ for $a, b,n  \geqslant 0$.

Then using $\gamma \theta <1$ together with the elementary inequalities of Lemma~\ref{lem:exp inequality}, we obtain for $x\geqslant \bar{x}$:
\begin{align}
F(x) &= \frac{2}{\sigma^2}\int_{y_0}^{x^{\frac{1}{n}}} v^{-\gamma \theta} {\rm e}^{\gamma v}\int_{y_0}^{v}y^{\gamma \theta -1}{\rm e}^{-\gamma y} dy\ dv \geqslant \frac{2}{\sigma^2}\int_{y_0}^{x^{\frac{1}{n}}} v^{-\gamma \theta} v^{\gamma \theta -1}{\rm e}^{\gamma v}\int_{y_0}^{v}{\rm e}^{-\gamma y} dy\ dv  \notag\\
&\geqslant \frac{2}{\gamma \sigma^2}\int_{y_0}^{x^{\frac{1}{n}}}v^{-1}{\rm e}^{\gamma v}\left({\rm e}^{-\gamma y_0}-{\rm e}^{-\gamma v}\right)dv 
\geqslant \frac{ {\rm e}^{-\gamma y_0}}{\gamma \sigma^2 } \int_{y_0 + \frac{\log(2)}{\gamma}}^{x^{\frac{1}{n}}}v^{-1}{\rm e}^{\gamma v}dv \notag \\
&\geqslant \frac{{\rm e}^{-\gamma y_0}}{\gamma \sigma^2 } x^{-\frac{1}{n}} \int_{y_0 + \frac{\log(2)}{\gamma}}^{x^{\frac{1}{n}}}{\rm e}^{\gamma v}dv
=\frac{{\rm e}^{-\gamma y_0}}{\gamma^2 \sigma^2 } x^{-\frac{1}{n}} \left({\rm e}^{\gamma x^{\frac{1}{n}}}-{\rm e}^{\gamma (y_0 + \frac{\log(2)}{\gamma})}\right) \notag \\
&\geqslant \frac{{\rm e}^{-\gamma y_0}}{2 \gamma^2 \sigma^2 } x^{-\frac{1}{n}} {\rm e}^{\gamma x^{\frac{1}{n}}}\label{eq:lower.bound.F positive} \\
&\geqslant  \frac{{\rm e}^{-\gamma y_0}}{2 \gamma^2 \sigma^2 } {\rm e}^{\frac{\gamma}{2} x^{\frac{1}{n}}}. \label{eq:lower.bound.F quadratic}
\end{align}
In view of \eqref{eq:lower.bound.F quadratic} and since $F$ is increasing, it follows that 
\begin{align}
H(y) &\leqslant  \left( \frac{2 \log(y)}{\gamma} + 2y_0 + \frac{4}{\gamma} \log(2 \gamma \sigma)\right)^n \leqslant \left( \frac{2 \log(y)}{\gamma} + 2y_0 + 8 \sigma \right)^n \notag \\
&\leqslant 4^n \gamma^{-n} \log(y)^n + 8^ny^n_0 + 32^n \sigma^n, \quad \mbox{for $y \geqslant F(\bar x)$.}
\label{eq:Peskir:H estimate after bar x}
\end{align}
Here, we have used in the last two steps the elementary inequalities $\log(x) \leqslant x$ for $x \in (0, \infty)$ and $(a + b)^n \leqslant 2^n (a^n+ b^n)$ for $a, b, n \geqslant 0$.
Moreover, for $y \in [0, F(\bar x)]$, using again that $F$ is increasing, we have
\begin{equation}
\label{eq:Peskir:H estimate before bar x}
H(y) \leqslant \bar x.
\end{equation}
Combining \eqref{eq:Peskir:H estimate after bar x} and \eqref{eq:Peskir:H estimate before bar x}, we obtain for all $y \in [0, \infty)$,
\begin{equation}
\label{eq:Peskir:H estimate}
H(y) \leqslant 4^n \gamma^{-n} \log(y\vee 1)^n + 1+ 8^n \big(y_0^n+ \gamma^{-2n} + 4^n \sigma^n\big).
\end{equation}

Finally, we derive an upper bound for  $\sup_{x\geqslant 1}\frac{F(x)}{x}\int_{x}^{\infty}\frac{dz}{F(z)} + 2$.
First, by \eqref{eq:lower.bound.F positive}, a change of variables, and Lemma \ref{lem:incomplete Gamma} (noting that $(\bar x)^{\frac{1}{n}} \gamma \geqslant 1$), we obtain for $x\geqslant \bar{x}$ that	
\begin{align}
\int_{x}^{\infty}\frac{dz}{F(z)} &\leqslant 2 \gamma^2 \sigma^2  {\rm e}^{\gamma y_0} \int_{x}^{\infty}z^{\frac{1}{n}}{\rm e}^{-\gamma z^{\frac{1}{n}}}dz = 2 n \gamma^2 \sigma^2  {\rm e}^{\gamma y_0}\int_{x^{\frac{1}{n}}}^{\infty}w^{n}{\rm e}^{-\gamma w}dw\notag\\
&\leqslant 6 n (n!) \gamma \sigma^2  {\rm e}^{\gamma y_0}x {\rm e}^{-\gamma x^{\frac{1}{n}}} \leqslant 6 (n+1)! \gamma \sigma^2  {\rm e}^{\gamma y_0}x {\rm e}^{-\gamma x^{\frac{1}{n}}}. \label{eq.sup.dom.ineq2}
\end{align}
Putting together Equations \eqref{eq.sup.dom.ineq1} and \eqref{eq.sup.dom.ineq2}, we obtain for $x\geqslant \bar{x}$,
\begin{equation}
\label{eq:Peskir:F integral estimate after bar x}
	\frac{F(x)}{x}\int_{x}^{\infty}\frac{dz}{F(z)}\leqslant \frac{12}{y_0} \left((n+1)!\right) \gamma^{-1}.
\end{equation} 
Using that $F$ is increasing, and the estimate \eqref{eq:Peskir:F integral estimate after bar x}, we obtain for $x \in [y_0, \bar x]$, 
\begin{align}
\frac{F(x)}{x}\int_{x}^{\infty}\frac{dz}{F(z)} &\leqslant  \frac{1}{y_0} \left(\int_{x}^{\bar x}\frac{F(x)}{F(z)} dz + F(\bar x)\int_{\bar x}^{\infty}\frac{dz}{F(z)}\right) \notag \\
&\leqslant \frac{1}{y_0} \left( \int_{x}^{\bar x} dz + \bar x \frac{F(\bar x)}{\bar x}\int_{\bar x}^{\infty}\frac{dz}{F(z)} \right) \notag \\
& \leqslant \frac{1}{y_0} \left(\bar x  + \bar x \frac{12}{y_0} \left((n+1)!\right) \gamma^{-1}\right).
\label{eq:Peskir:F integral estimate before bar x}
\end{align}
Combining \eqref{eq:Peskir:F integral estimate before bar x} and \eqref{eq:Peskir:F integral estimate after bar x} gives
\begin{equation}
\label{eq:Peskir:F integral estimate} 
\sup_{x\geqslant 1}\frac{F(x)}{x}\int_{x}^{\infty}\frac{dz}{F(z)}+2 \leqslant 2 + \frac{\bar x}{y_0} \left(1 + \frac{12}{y_0} \left((n+1)!\right) \gamma^{-1} \right).
\end{equation}
Now, the result follows from \eqref{eq:Peskir.ineq}, \eqref{eq:Peskir:tilde H estimate},  \eqref{eq:Peskir:H estimate}, and \eqref{eq:Peskir:F integral estimate}.
\end{proof}

The following elementary estimates are used in the proof of Proposition~\ref{prop:CIR process}:

\begin{lemma}
\label{lem:exp inequality}
Let $\gamma > 0$ and $y \geqslant 0$. Then:
\begin{align}
\label{eq:lem:exp inequality:quadratic}
\frac{1}{x}\exp(\gamma x) &\geqslant \exp\left(\frac{\gamma}{2} x\right), \quad \text{for all } x \geqslant \frac{8}{\gamma^2}, \\
\label{eq:lem:exp inequality:positive}
\exp(\gamma x) - \exp( \gamma y) &\geqslant \frac{1}{2} \exp(\gamma x), \quad \text{for all } x \geqslant y + \frac{\log(2)}{\gamma}, \\
\label{eq:lem:exp inequality:negative}
\exp(-\gamma y)  - \exp(-\gamma x) &\geqslant \frac{1}{2} \exp(-\gamma y), \quad \text{for all } x \geqslant y + \frac{\log(2)}{\gamma}.
\end{align}
\end{lemma}

\begin{lemma}
\label{lem:incomplete Gamma}
Let $\gamma > 0$, $n \in \mathbb{N}$ and $y \geqslant \frac{1}{\gamma}$. Then:
\begin{equation*}
\int_{y}^\infty x^n \exp(-\gamma x) d x \leqslant 3 (n!) \gamma^{-1 } y^n \exp(-\gamma y).
\end{equation*}
\end{lemma}

\begin{proof}
Set $w := \gamma y \geqslant 1$. Then by a change of variables,
\begin{equation*}
\int_{y}^\infty x^n \exp(-\gamma x) d x =  \gamma^{-1 -n} \int_w^\infty z^n \exp(-z) dz.
\end{equation*}
Moreover, integration by parts (and induction) together with $w \geqslant 1$ give
\begin{align*}
\int_w^\infty z^n \exp(-z) dz &= \sum_{k =0}^n w^{n-k} \frac{n!}{(n-k)!} \exp(-w) \leqslant w^n \exp(-w) \sum_{k =0}^n \frac{n!}{(n-k)!} \\
&\leqslant 3 (n!)  w^n \exp(-w). \qedhere
\end{align*}
\end{proof}

Using Proposition~\ref{prop:CIR process}, we now establish a moment estimate for the supremum of $X^\epsi/\epsi$ that is used in Section~\ref{section.concatenation}.

\begin{lemma}
\label{lem.max.inequality}
Suppose that Assumption \ref{ass.novikov} is satisfied and there is $\kappa \in (0, 1)$ such that 
\begin{equation}
	\label{eq:lem.max.inequality:boundedness condition}
L_t, M_t \in \left[\kappa, \frac{1}{\kappa} \right], \quad t \in [0, T].
\end{equation}
Set 
\begin{equation}
	\label{eq:lem.max.inequality:z0}
z_0 := 2 \limsup_{\epsi \to 0} (x^\epsi_0)^2/\epsi^2 \vee 1.
\end{equation}
Then, for $n\in \bN$ and $\epsi > 0$:
\begin{align}
\bE\left[\max_{0\leqslant t \leqslant
T}\left(\frac{X^\epsi_t}{\epsi}\right)^{\!n}\right] &\leqslant \bE_{\bP}\left[\frac{d\bP}{d\bQ}\right]^{\frac{1}{2}}\bigg(\sqrt{C^1(z_0,a\kappa^2,2,n)} +\sqrt{C^2(z_0,a\kappa^2,n)} \times \notag \\
&\qquad\qquad \qquad  \times \bE\bigg[\left(\frac{d\bQ}{d\bP}\right)^2\bigg]^{\frac{1}{2}}\bE\bigg[\log\Big(\int_{0}^{T}\epsi^{-2}c_t dt\vee 1\Big)^{2n} \bigg]^{\frac{1}{4}}\bigg) \notag \\
\label{pre.uniform.moment} & =O\left(\log\left(\frac{1}{\epsi}\right)^{\frac{n}{2}}\right).
\end{align}
Here, the functions $C^1$ and $C^2$ are defined as in Proposition~\ref{prop:CIR process}, and $a$ is a positive constant such that  $xg(x)\geqslant a\left|x\right|^2-\frac{1}{4}$ for all $x\in\bR$.
\end{lemma} 

\begin{proof}
Using Proposition~\ref{prop:CIR process}, we show below that
\begin{equation}
\label{eq:pre.estimate.barX.Q}
\bE_{\bQ}\left[\max_{0\leqslant t \leqslant T}\left(\frac{X^\epsi_t}{\epsi}\right)^{\!2n}\right] \leqslant  C^1(z_0,a\kappa^2,2,n) + C^2(z_0,a\kappa^2,n) \bE_{\bQ}\left[\log(\xi^\epsi\vee 1)^n \right],
\end{equation}
where $\xi^\epsi := \int_{0}^{T}\epsi^{-2}c_t dt$ is as in \eqref{eq:xi.epsi.m}.
Then by the Cauchy--Schwarz inequality, the elementary inequality $\sqrt{b+c} \leqslant \sqrt{b} + \sqrt{c}$ for $ b, c \geqslant 0$, and again the Cauchy--Schwarz inequality, we obtain
\begin{align*}
&\bE\left[\max_{0\leqslant t \leqslant T}\left(\frac{X^\epsi_t}{\epsi}\right)^{n}\right] \leqslant \bE_{\bQ}\bigg[\bigg(\frac{d\bP}{d\bQ}\bigg)^2\bigg]^{\frac{1}{2}} \bE_{\bQ}\left[\max_{0\leqslant t \leqslant T}\left(\frac{X^\epsi_t}{\epsi}\right)^{\!2n}\right]^{\frac{1}{2}}\\
&\qquad\leqslant\bE\bigg[\frac{d\bP}{d\bQ}\bigg]^{\frac{1}{2}} \left(\sqrt{C^1(z_0,a\kappa^2,2,n)}+\sqrt{C^2(z_0,a\kappa^2,n)}\bE_{\bQ}\left[\log(\xi^\epsi\vee 1)^n \right]^{\frac{1}{2}}\right) \\
&\qquad\leqslant\bE\bigg[\frac{d\bP}{d\bQ}\bigg]^{\frac{1}{2}} \left(\sqrt{C^1(z_0,a\kappa^2,2,n)}+\sqrt{C^2(z_0,a\kappa^2,n)}\bE\bigg[\left(\frac{d\bQ}{d\bP}\right)^2\bigg]^{\frac{1}{2}} \bE\left[\log(\xi^\epsi\vee 1)^{2n} \right]^{\frac{1}{4}}\right) \\
&\qquad = O\left(\log\left(\frac{1}{\epsi}\right)^{\frac{n}{2}}\right).
\end{align*}
It remains to prove \eqref{eq:pre.estimate.barX.Q}. The definition of the rescaled and time-changed process \eqref{eq:relation.X} gives
\begin{equation}
\label{pre.maximum.rescaled}
\bE_{\bQ}\left[\max_{0\leqslant t \leqslant T}\left(\frac{X^\epsi_t}{\epsi}\right)^{\!2n}\right] =\bE_{\bQ}\left[\max_{0\leqslant \xi\leqslant \xi^\epsi}\left(\widetilde{X}^\epsi_\xi\right)^{\!2n}\right].
\end{equation}
Under $\bQ$, $(\widetilde{X}^{\epsi}_\xi)^2$ satisfies the SDE
\begin{equation*}
d\big(\widetilde{X}^{\epsi}_\xi\big)^2 =  \left(1-2L_{u^\epsi_\xi}\widetilde{X}^{\epsi}_\xi g\left(M_{u^\epsi_\xi}\widetilde{X}^{\epsi}_\xi\right)\right)\one_{\left\{\xi\leqslant \xi^\epsi\right\}}d\xi+2 \sqrt{\big(\widetilde{X}^{\epsi}_\xi\big)^2} \one_{\left\{\xi\leqslant \xi^\epsi\right\}}d\widetilde{B}^{\epsi,\bQ}_\xi,
\end{equation*}
where
\begin{equation*}
\widetilde{B}^{\epsi,\bQ}_\xi=\int_{0}^{\xi}\text{sgn}\big(\widetilde{X}^{\epsi}_y\big)d\widetilde{W}^{\epsi,\bQ}_y
\end{equation*}
is a $\bQ$-Brownian motion stopped at $\xi^\epsi$. Furthermore, the growth condition of $g$ \eqref{eq:growth.g} implies that there exist a constant $a$ such that $xg(x)\geqslant a\left|x\right|^2-\tfrac{1}{4}$. In view of \eqref{eq:lem.max.inequality:boundedness condition}, 
\begin{align*}
1-2L_{u^\epsi_\xi}x g\left(M_{u^\epsi_\xi}x\right)&\leqslant 1-2\frac{L_{u^\epsi_\xi}\wedge M_{u^\epsi_\xi}}{M_{u^\epsi_\xi}}M_{u^\epsi_\xi}xg\left(M_{u^\epsi_\xi} x\right) \leqslant 1-2\frac{L_{u^\epsi_\xi}\wedge M_{u^\epsi_\xi}}{M_{u^\epsi_\xi}}\left(aM^2_{u^\epsi_\xi}\left|x\right|^2-\frac{1}{4}\right)\\
&\leqslant \frac{3}{2}-2a\kappa^2 \left|x\right|^2,\quad\mbox{a.s., for $x\in\bR$.}
\end{align*}
Using the comparison result established in Lemma~\ref{lem:app:comparison}, we have $\bQ\left[\left(\widetilde{X}^\epsi_\xi\right)^2\leqslant Z_\xi, \mbox{for all }0\leqslant\xi\leqslant \xi^\epsi\right]$, where $Z$ is the solution of the SDE
\begin{equation*}
d Z_t = \left(\frac{3}{2} -  2a\kappa^2 Z_t \right) dt + 2 \sqrt{Z_t} d\widetilde{B}^{\epsi,\bQ}_t, \quad Z_0 = z_0,
\end{equation*}
where $z_0$ is defined in \eqref{eq:lem.max.inequality:z0}.
This is a special case of the equation~\eqref{eq:prop:CIR process:CIR SDE} studied in Proposition \ref{prop:CIR process} with $\nu=2a\kappa^2$, $\theta=\frac{3}{4a\kappa^2}$, $\sigma=2$ and $\gamma=a\kappa^2$. Combining \eqref{pre.maximum.rescaled} and \eqref{pre:CIR process}, we finally obtain the asserted estimate:
\begin{align*}
\bE_{\bQ}\left[\max_{0\leqslant t \leqslant T}\left(\frac{X^\epsi_t}{\epsi}\right)^{\!2n}\right] &\leqslant\bE_{\bQ}\left[\max_{0\leqslant \xi \leqslant \xi^\epsi} Z_\xi^{n}\right] \leqslant C^1(z_0,a\kappa^2,2,n)+C^2(z_0,a\kappa^2,n)\bE_{\bQ}\left[\log\left(\xi^\epsi \vee 1\right)^{n}\right].
\end{align*}
This completes the proof.
\end{proof}

\section{A Comparison Result for SDEs}
\label{app:comparison}

In this appendix, we establish a comparison result for one-dimensional SDEs that is used at various points in the proofs of our main results. It extends the standard argument from \cite[Proposition~5.2.18]{karatzas.shreve.91} to the case of random initial conditions as well as drift and diffusion coefficients that are not globally Lipschitz:

\begin{lemma}
	\label{lem:app:comparison}
	Let $(\Omega, \cF, \bF = (\cF_t)_{t \geqslant 0}, \bP)$ be a filtered probability space satisfying the usual conditions. Let $\tau$ be an $\bF$-stopping time taking values in $[0, \infty]$, $(c_t)_{t \geqslant 0}$ an $\bR^d$-valued, locally bounded and $\bF$-adapted process,  $(W_t)_{t \geqslant 0}$ a continuous $\bF$-adapted process that is a standard Brownian motion on $\llbracket 0,  \tau \rrbracket$, and $\Omega_0 \in \cF_0$. Suppose $(Y^{(1)}_t)_{t \geqslant 0}$ and $(Y^{(2)}_t)_{t \geqslant 0}$ are continuous $\bF$-adapted processes that satisfy the SDEs
	\begin{equation}
	\label{eq:lem:app:comparison}
	d Y^{(i)}_t = b^{(i)}(\omega, t, Y^{(i)}_t) \one_{\{t \leqslant \tau\}}dt + h\big(c_t,Y^{(i)}_t\big) \one_{\{t \leqslant \tau\}} d W_t, \quad Y^{(i)}_0 = y^{(i)}_0,
	\end{equation}
	where $y^{(i)}_0 \geqslant 0$ is $\cF_0$-measurable and $b^{(i)}$ is $\cF$-predictable for $i=1,2$, and $h$ is $1/2$-H\"older continuous in its second variable:
	\begin{equation*}
		\big|h\big(c,x\big)-h\big(c,y\big)\big|	\leqslant K(c)\sqrt{|x-y|},
	\end{equation*}
	where $K$ is a locally bounded function from $\bR^d$ to $\bR_+$. Set
	\begin{equation*}
	\widetilde \Omega^\tau_0 := \{(\omega,t) \in \Omega \times [0, \infty): \omega \in \Omega_0, t \in [0, \tau(\omega)]\}
	\end{equation*}
	and  assume that:
	\begin{enumerate}
		\item $b^{(1)}(\omega,t,y) \leqslant b^{(2)}(\omega,t,y)$ \text{ for all } $(\omega, t) \in \widetilde \Omega_0^\tau$ and $y \in \bR_+$;
		\item $y^{(1)}_0(\omega) \leqslant y^{(2)}_0(\omega)$ for all $\omega \in \Omega_0$;
		%\item Either $b^{(1)}$ or $b^{(2)}$ is locally Lipschitz in $x$, uniformly in $(\omega,t) \in \widetilde \Omega^\tau_0$.
		\item Either $b^{(1)}$ or $b^{(2)}$ is locally one-sided Lipschitz in $x$, uniformly in $(\omega,t) \in \widetilde \Omega^\tau_0$, i.e.
		\begin{equation*}
			\forall x_1\in \bR, \exists r_{x_1}: \forall x_2\in[x_1-r_{x_1},x_1],~~ b^{(i)}(\omega,t,x_1)-b^{(i)}(\omega,t,x_2) \leqslant K_{x_1}(x_1-x_2),
	\end{equation*}
	for some $K_{x_1}>0$ and all $(t,\omega)\in\Omega^\tau_0$.
	\end{enumerate}
	Then 
	\begin{equation}
		\label{eq:lem:app:comparison:result}
	\bP\left[\one_{\Omega_0} Y^{(1)}_t\leqslant \one_{\Omega_0} Y^{(2)}_t , \text{ for all } t \geqslant 0 \right] = 1.
	\end{equation}
\end{lemma}

\begin{proof}
We may assume without loss of generality that $b^{(1)}$ satisfies (iii). By a standard localization argument, we may  assume that $b^{(1)}$ is globally one-sided Lipschitz in $x$, uniformly in $(\omega,t) \in \widetilde \Omega^\tau_0$, with Lipschitz constant $K > 0$. By a further localization argument, we may assume that $c$, the function $K$, $Y^{(1)}$ and $ Y^{(2)}$ are bounded by a constant $L > 0$.%\footnote{Since $Y^{(1)}, Y^{(2)}$ and $c$ are continuous, they are locally bounded. Let $(\tau_n)_{n \in \bN}$ be a nondecreasing sequence of stopping times converging almost surely to $\infty$ such that for each $n \in \bN$, the stopped processes $c^{\tau_n}$, $c^{\tau_n} Y^{(1),\tau_n}$ and $c^{\tau_n} Y^{(2),\tau_n}$ are bounded. Note that $Y^{(i),\tau_n}$ satisfies the SDE \eqref{eq:lem:app:comparison} with $\tau$ replaced by $\tau \wedge \tau_n$, and conditions (i) -- (iii) are satisfied for $\widetilde \Omega^{\tau \wedge \tau_n}_0$. Moreover $Y^{(i)}$ and $Y^{(i),\tau_n}$ coincide on $\llbracket 0, \tau_n \rrbracket$ for each $n \in \bN$ and $i \in \{1, 2\}$. Thus, it suffices to establish  \eqref{eq:lem:app:comparison:result} for $Y^{(i),\tau_n}$ for each $n \in \bN$.}
	
By the construction in the proof of \cite[Proposition 5.2.13]{karatzas.shreve.91}, there exists a nondecreasing sequence $(\varphi_n)_{n \in \bN \setminus \{0\}}$ of nonnegative $C^2$ functions such that:
	\begin{enumerate}
		\item [(a)] for each $n$, $\varphi_n$ is supported on $[a_n, \infty)$ for some constant $a_n > 0$, and satisfies $0 \leqslant \varphi^\prime_n(x) \leqslant 1$ and  $0 \leqslant \varphi^{\prime\prime}_n(x) \leqslant \frac{2}{n x}$ for $x > 0$;
		\item [(b)] $\lim_{n \to \infty} \varphi_n(x) = x^+$ for $x \in \bR$.
	\end{enumerate}
	Fix $t > 0$ and $n \in \bN \setminus \{0\}$. Set $\Delta_t :=  Y^{(1)}_t - Y^{(2)}_t$. By It\^o's formula, the $1/2$-H\"older continuity of $h$, the fact that $0 \leqslant \varphi^{\prime\prime}_n(x) \leqslant \frac{2}{n x}$ for $x > 0$ by Property~(a), and the assumption that $K \leqslant L$, we obtain
	\begin{align}
	\varphi_n(\Delta_t) &= \varphi_n(\Delta_0) + \int_0^{t} \varphi_n^\prime(\Delta_s) \Big(b^{(1)}(\omega, s, Y^{(1)}_s) - b^{(2)}(\omega, s, Y^{(2)}_s) \Big) \one_{\{s \leqslant \tau\}} ds \notag \\
	&\qquad\qquad\;\;+ \frac{1}{2} \int_0^{t} \varphi_n^{\prime\prime}(\Delta_s) \left(h\big(c_s,Y^{(1)}_s\big)-h\big(c_s,Y^{(2)}_s\big)\right)^2\one_{\{s \leqslant \tau\}} ds \notag\\
	&\qquad\qquad\;\;+ \int_0^{t} \varphi_n^\prime(\Delta_s) \left(h\big(c_s,Y^{(1)}_s\big)-h\big(c_s,Y^{(2)}_s\big)\right) \one_{\{s \leqslant \tau\}}dW_s\notag \\
	&\leqslant  \varphi_n(\Delta_0) + \int_0^{t}  \varphi_n^\prime(\Delta_s) \Big(b^{(1)}\big(\omega, s, Y^{(1)}_s\big) - b^{(2)}\big(\omega, s, Y^{(2)}_s\big) \Big) \one_{\{s \leqslant \tau\}}ds + \frac{t L}{n}\notag\\
	&\qquad\qquad\;\;+ \int_0^{t \wedge \tau} \varphi_n^\prime(\Delta_s) \left(h\big(c_s,Y^{(1)}_s\big)-h\big(c_s,Y^{(2)}_s\big)\right) \one_{\{s \leqslant \tau\}} dW_s.
	\label{eq:pf:lem:app:comparison:Ito}
	\end{align}
	Now multiply the inequality \eqref{eq:pf:lem:app:comparison:Ito} with $\one_{\Omega_0}$, use that $\Delta_0 \leqslant 0$ on $\Omega_0$ by Assumption (ii), and note that $\int_0^{\cdot} \varphi_n^\prime(\Delta_s) h(c_s,Y^{(i)}_s) dW_s$ is a  martingale for $i \in \{1, 2\}$ since $K(c)$, $Y^{(1)}$ and $Y^{(2)}$ are bounded and $0 \leqslant \varphi^\prime_n(x) \leqslant 1$ for $x > 0$ by Assumption~(a). Also taking into account Assumption (i), that $\varphi'(x)=0$ on $\bR_-$ and that $b^{(1)}$ is one-sided Lipschitz in $x$, uniformly in $(\omega, t) \in \widetilde \Omega^\tau_0$, with constant $K > 0$, it follows that
	\begin{align*}
	\bE\left[\one_{\Omega_0} \varphi_n(\Delta_t)\right] \leqslant 0 &+ \bE\left[\int_0^t \one_{\{\Delta_s > 0\}} \varphi_n^\prime(\Delta_s) \one_{\Omega_0} \one_{\{s \leqslant \tau\}} \Big(b^{(1)}(\omega, s, Y^{(1)}_s) - b^{(1)}(\omega, s, Y^{(2)}_s) \Big) ds\right] \\
	&+ \bE\left[\int_0^t  \varphi_n^\prime(\Delta_s) \one_{\Omega_0} \one_{\{s \leqslant \tau\}} \Big(b^{(1)}(\omega, s, Y^{(2)}_s) - b^{(2)}(\omega, s, Y^{(2)}_s) \Big) ds\right] + \frac{t L}{n} +  0\\
	&\leqslant  K \bE\left[\int_0^t \one_{\Omega_0} \big|\Delta_s \big| \one_{\{\Delta_s > 0\}}ds \right] + 0 + \frac{t L}{n}.
	\end{align*}
	Letting $n \to \infty$, monotone convergence, Property (b) and Fubini's theorem give 
	\begin{equation*}
	\bE\left[\one_{\Omega_0} \big(\Delta_t \big)^+\right] \leqslant  K \int_0^t \bE\left[ \one_{\Omega_0} \big(\Delta_s\big)^+\right]ds.
	\end{equation*}
	Now apply Gronwall's inequality to the function 
	\begin{equation*}
	h(s) := \bE\left[\one_{\Omega_0} \big(\Delta_s\big)^+\right] \geqslant 0.
	\end{equation*}
	This yields $h(s) = 0$ for $s \in [0, t]$, and in turn $\one_{\Omega_0}  Y^{(1)}_s \leqslant \one_{\Omega_0}  Y^{(2)}_s$ $\bP$-a.s.\ for $s \in[0,t]$. The result now follows from the continuity of the paths of $Y^1$ and $Y^2$.
\end{proof}
\begin{remark}
Note that this rather general comparison result accommodates, in particular, ``bang-bang''-controlled SDEs of the form
\begin{equation*}
	dX^\epsi_t =-\frac{1}{\epsi}\text{sgn}(X^\epsi_t)dt+dW_t.
\end{equation*}
This is because our arguments only require the drift functional to be one-sided Lipschitz, as was kindly pointed out to as by one of the anonymous reviewers.
\end{remark}

\section{Auxiliary Results}
\label{app:auxiliary}

\subsection{Existence results for SDEs}
The following two strong existence results are somewhat nonstandard because the volatility functions are not locally Lipschitz at $0$ and, in the second result, the drift and volatility coefficient are not necessarily Markov. 

\begin{lemma}
	\label{lem:app:existence:SDE}
	Let $(\Omega, \cF, \bF = (\cF_t)_{t \geqslant 0}, \bP)$ be a filtered probability space satisfying the usual conditions, $y_0 \geqslant 0$ be an $\cF_0$-measurable random variable, and $(W_t)_{t \geqslant 0}$ an $\bF$-Brownian motion. Moreover, let $b: \bR \to \bR$ be locally Lipschitz, nonnegative on $\bR^+$, odd and null at zero. Then the SDE
	\begin{equation}
	\label{eq:lem:app:existence:SDE}
	d Y_t =  \left(1- 2 \sqrt{Y_t} b\left(\sqrt{Y_t}\right) \right) dt + 2\sqrt{Y_t} dW_t, \quad Y_0 = y_0,
	\end{equation}
	has a unique strong solution.
\end{lemma}

\begin{proof}
	By \cite[Corollary 5.3.23]{karatzas.shreve.91}, it suffices to show that  weak existence and pathwise uniqueness hold for the SDE \eqref{eq:lem:app:existence:SDE}.
	To establish weak existence, consider the SDE
	\begin{equation}
	d Z_t= - b(Z_t) dt + d W_t, \quad Z_0 = \sqrt{y_0}.
	\end{equation}
It follows by the same argument as in Proposition \ref{prop:sde} that $Z$ has a unique strong solution. Now set $Y := Z^2$ and define the Brownian motion $B$ by
\begin{equation*}
B_t = \int_0^t\text{sgn}(Z_t) d W_t.
\end{equation*}
Then $Y$ satisfies the SDE
\begin{equation*}
d Y_t = \left(1- 2 \sqrt{Y_t} b(\sqrt{Y_t}) \right) dt + 2\sqrt{Y_t} dB_t, \quad Y_0 = y_0,
\end{equation*}
and so \eqref{eq:lem:app:existence:SDE} has a weak solution.

Pathwise uniqueness follows from Lemma \ref{lem:app:comparison} with $\tau =+\infty$, $\Omega_0 = \Omega$ and $b^1(\omega, t, y) = b^2(\omega, t, y) = 1 - 2 \sqrt{y} b(\sqrt{y}) $ and $y^1_0 = y^2_0$. Note that $x\mapsto \sqrt{x}b(\sqrt{x})$ is locally Lipschitz under the assumptions on $b$.
\end{proof}

\begin{lemma}
	\label{lem:app:existence:SDE:2}
	Let $(\Omega, \cF, \bF = (\cF_t)_{t \in [0, T]}, \bP)$ be a filtered probability space satisfying the usual conditions, $y_0 \geqslant 0$ be an $\cF_0$-measurable random variable, and $(W_t)_{t \in [0, T]}$ an $\bF$-Brownian motion. Moreover, let $n \in \mathbb{N} \setminus \{0\}$ and $(c_t)_{t \in [0, T]}$ be a positive, continuous and $\bF$-adapted process satisfying 
	\begin{equation}
		\label{eq:lem:app:existence:SDE:n:mart}
\bE\left[\int_0^T c^{2n}_s ds\right] < \infty.
	\end{equation}
	Then the SDE
	\begin{equation}
	\label{eq:lem:app:existence:SDE:2}
	d Y_t =  c_t dt + 2\sqrt{Y_t} \sqrt{c_t} dW_t, \quad Y_0 = y_0,
	\end{equation}
	has a unique strong solution. Moreover,
		\begin{equation}
		\label{eq:lem:app:existence:SDE:n:int}		
\bE\left[\sup_{t \in [0, T]} (Y_t)^{2n} \right] < \infty.
	\end{equation}
\end{lemma}

\begin{proof}
Existence of a unique strong solution follows by a time change argument. Indeed, use the time change from Lemma \ref{lem:time.change} with $\epsi = 1$ and just write $u_\xi$ instead of $u_\xi^1$. Then writing $\widetilde W_\xi$ instead of $\widetilde W^1_\xi$ and setting $\widetilde Y_\xi := Y_{u_\xi}$, it suffices to show that the SDE
\begin{equation*}
dY_\xi = d\xi + 2 \sqrt{Y_\xi} d \widetilde W_\xi, \quad \widetilde Y_0 = y_0
\end{equation*}
has a unique strong solution. This is clear as this is the SDE satisfied by the square of a one-dimensional Bessel process started at $y_0$; cf.~\cite[Definition XI.1.1]{revuz.yor.99} and note that this result extends to non-trivial initial condition by virtue of \cite[Corollary 5.3.23]{karatzas.shreve.91}.

We proceed to derive \eqref{eq:lem:app:existence:SDE:n:int}. Define the process $Z$ by
	\begin{equation}
d Z_t=  \sqrt{c_t} d W_t, \quad Z_0 = \sqrt{y_0}.
\end{equation}
Then $Z$ is a martingale with finite $4n$-th moments by \eqref{eq:lem:app:existence:SDE:n:mart}. In particular by the Burkholder-Davis-Gundy inequality, $\bE\left[\sup_{t \in [0, T]} Z_t^{4n}\right] < \infty$.
Define the process $Y$ by $Y := Z^2$. Then $Y$ is a weak solution of the SDE \eqref{eq:lem:app:existence:SDE:2}. Now \eqref{eq:lem:app:existence:SDE:n:int} follows from the fact that $\sup_{t \in [0, T]} Y_t^{2n} = \sup_{t \in [0, T]} Z_t^{4n}$ and uniqueness in law of any strong or weak solution to \eqref{eq:lem:app:existence:SDE:2}.
\end{proof}

\subsection{An Ergodic Result} 

This section contains an ergodic theorem for the one-dimensional diffusions defined in (\ref{sde_y.epsilon.-}), which is used in the proof of Lemma~\ref{lem:ergodic elementary}. For constants $l,m > \delta>0$ and $y \geqslant 0$, consider the following two SDEs on some filtered probability space:
\begin{equation}\label{eq:SDEpm}
dY_t^{\pm} = \left(1-2(l \mp\delta)\sqrt{Y^\pm_t}g\left((m\mp \delta)\sqrt{Y^\pm_t}\right)\right)dt+2\sqrt{Y^\pm_t}dB_t, \quad Y_0=y,
\end{equation}
for a standard Brownian motion $(B_t)_{t \geqslant 0}$. (Existence and uniqueness follow from Lemma~\ref{lem:app:existence:SDE}.)

\begin{lemma} 
	\label{lemma.scale.speed.integrability}
	The diffusions $Y^{+}$, $Y^{-}$ are recurrent; their speed measures are finite and have the following densities:
	\begin{align*}
	\nu^{\pm}(y) &=\frac{1}{2}y^{-\frac{1}{2}}\exp\Big(-2\frac{l \mp \delta}{m\mp\delta}G\left((m \mp \delta)\sqrt{y}\right)+2\frac{l \mp \delta}{m\mp\delta}G(m \mp \delta)\Big) \one_{\{y \geqslant 0\}}.
	\end{align*}
\end{lemma}

\begin{proof}
	To prove that $Y^{+}$, $Y^{-}$ are recurrent, first note that this property only depends on the respective laws. Whence, it is enough to verify it for any weak solutions of the SDEs~\eqref{eq:SDEpm}. Such solutions are given by the squares of the solutions of the SDE~\eqref{eq:SDE.proof.existence} with constant coefficients $L=l\pm\delta$, $M=m\pm\delta$, and without stopping. To prove recurrence of $Y^{\pm}$ on $\mathbb{R}_+$ it is in turn sufficient to verify recurrence of these solutions on $\mathbb{R}$, which follows from \cite[Proposition 5.5.22(a)]{karatzas.shreve.91}.
	
	To compute the speed measures of $Y^{+}$, $Y^{-}$, first note that the respective scale functions (cf.~\cite[Equation (5.42)]{karatzas.shreve.91} are
	\begin{align*}
	p^\pm(y) &= \int_{1}^{y}\exp\left(-2\int_{1}^{x}\frac{1 -2(l \mp \delta)\sqrt{z}g\left((m \mp \delta)\sqrt{z}\right)}{4z}dz\right)dx\\
	&=\int_{1}^{y}x^{-\frac{1}{2}}\exp\left(2\frac{(l \mp \delta)}{(m \mp \delta)}G\left((m \mp \delta)\sqrt{x}\right)-2\frac{(l \mp \delta)}{(m \mp\delta)}G(m \mp\delta)\right)dx, \quad \mbox{for $y \geqslant 0$,}
	\end{align*}
	where $G(x)=\int_{0}^{x}g(y)dy$. The asserted formulas for the densities of the corresponding speed measures in turn follows directly from the definition \cite[Equation (5.51)]{karatzas.shreve.91}. Finiteness follows from an elementary integration near zero and the growth condition~\eqref{eq:growth.g} for the function $g$ near infinity. 
\end{proof}

Lemma~\ref{lemma.scale.speed.integrability}, the ergodic theorem as in \cite[Section II.35]{borodin.salminen.02}, the growth condition for the functions $f$ and $g$, and a change of variable in turn yield the following ergodic limits: %\texttt{Note again that the plus and minus part are not symmetric.}

\begin{lemma}
	\label{lemma.limit.epsi.ergodic}
	Suppose $l,m > \delta>0$. Then, for any $k \geqslant 0$:
	\begin{align*}
\lim_{x\to\infty}\frac{1}{x}\int_{0}^{x}f\left((k + \delta) \sqrt{Y^+_t}\right) dt &= \frac{\int_{\bR_+}f\left(\frac{(k + \delta)}{m-\delta}y\right)\exp\left(-2\frac{l-\delta}{m-\delta}G(y)\right)dy}{\int_{\bR_+}\exp\left(-2\frac{l-\delta}{m-\delta}G(y)\right)dy} \quad\mbox{a.s.} \\
\lim_{x\to\infty}\frac{1}{x}\int_{0}^{x}f\left(k (1 - \delta) \sqrt{Y^{-}_t}\right) dt &= \frac{\int_{\bR_+}f\left(\frac{k(1- \delta)}{m+\delta}y\right)\exp\left(-2\frac{l+\delta}{m+\delta}G(y)\right)dy}{\int_{\bR_+}\exp\left(-2\frac{l+\delta}{m+\delta}G(y)\right)dy} \quad\mbox{a.s.}
	\end{align*}
\end{lemma}

\subsection{A Result from Measure Theory}
The following result from measure theory is used in the proof of Theorem \ref{theorem}:

\begin{lemma}
	\label{lemma.abstract}
	Let $(a^{\epsi,1}_t)_{t \in [0, T]}$ be a family of product-measurable processes indexed by $\epsi\in(0,\frac{1}{2})$ and $(a^2_t)_{t \in [0, T]}$ a product measurable-process. Suppose that, for ${\mathrm{Leb}}_{|[0,T]}$-a.e.~$t \in [0, T]$, $a^{\epsi,1}_t$ converges in probability to $a^2_t$ as $\epsi \to 0$.
	Moreover, assume that $|a^{\epsi,1}|^p$ is uniformly integrable (as a family indexed by $\epsi$) with respect to $\bP\otimes\mathrm{Leb}_{|[0,T]}$ and that $a^2 \in L^p (\bP\otimes\mathrm{Leb}_{|[0,T]})$ for some $p \geqslant 1$. Then $\int_{0}^{\cdot}a^{\epsi,1}_sds\to \int_{0}^{\cdot}a^2_sds$ in $\cS^p([0,T])$.
\end{lemma}
\begin{proof}
	First, we show that $a^{\epsi,1}$ converges to $a^2$ in measure under $\bP\otimes\text{Leb}_{|[0,T]}$. Indeed, for fixed $t \in [0, T]$, convergence in probability of $a^{\epsi,1}_t $ to $a^2_t$ is equivalent to
	\begin{equation*}
	\bE\left[\left|a^2_t-a^{\epsi,1}_t\right|\wedge 1\right]\to 0~~\mbox{as}~~\epsi\to 0.
	\end{equation*}
	Thus, by Fubini's theorem and dominated convergence, we obtain
	\begin{equation*}
	\bE\left[\int_{0}^{T}\left(\left|a^2_t-a^{\epsi,1}_t\right|\wedge 1\right) dt\right]\to 0~~\mbox{as}~~\epsi\to 0,
	\end{equation*}
	which is equivalent to convergence in measure under $\bP\otimes\text{Leb}_{|[0,T]}$ of $a^{\epsi,1}$ to $a^2$.
	
	Next \cite[Proposition 4.12]{kallenberg.02} implies that $a^{\epsi,1} \to a^2$ in $L^p(\bP\otimes\text{Leb}_{|[0,T]})$. The assertion in turn follows from Jensen's inequality via
	\begin{align*}
	\limsup_{\epsi \to 0}	\bE\left[\sup_{t\in[0,T]}\left|\int_{0}^{t}a^{\epsi,1}_sds-\int_{0}^{t}a^{2}_sds\right|^p\right]&\leqslant T^{p-1}\ \limsup_{\epsi \to 0} \bE\left[\int_{0}^{T}\left|a^{\epsi,1}_s-a^{2}_s\right|^p ds\right] = 0.
	\end{align*}
\end{proof}

\subsection{An Integrability Result}
The following result is used in the reduction to bounded coefficients in Section \ref{sec:reduction}.
\begin{lemma}
\label{lem:app:integrability estimate}
Suppose that Assumption \ref{ass.integrability} is satisfied. Then:
\begin{equation}
\label{eq:lem:app:integrability estimate}
		\bE\Bigg[\int_0^T\Bigg(H_s \frac{\int_{\bR} f\big(\frac{K_s}{M_s}y\big) \exp\big(-2 \frac{L_s}{M_s} G(y)\big)dy}{\int_{\bR}\exp\big(-2 \frac{L_s}{M_s} G(y)\big)dy}\Bigg)^{\!p(1+\eta)}ds\Bigg] < \infty.
		\end{equation}
\end{lemma}

\begin{proof}
We start by estimating the fraction appearing in \eqref{eq:lem:app:integrability estimate}.  By the growth condition \eqref{eq:condition.f} for $f$ and the fact that $G$ is even (because $g$ is odd), it follows that
\begin{equation}
\label{eq:pf:lem:app:integrability estimate:fraction}
\frac{\int_{\bR} f\left(\frac{K_s}{M_s}y\right) \exp\left(-2 \frac{L_s}{M_s} G(y)\right)dy}{\int_{\bR}\exp\left(-2 \frac{L_s}{M_s} G(y)\right)dy} \leqslant C_f + C_f \left(\frac{K_s}{M_s}\right)^{q^\prime} \frac{\int_0^\infty y^{q^\prime} \exp\left(-2 \frac{L_s}{M_s} G(y)\right)dy}{\int_0^\infty \exp\left(-2 \frac{L_s}{M_s} G(y)\right)dy}.
\end{equation}
We proceed to estimate the numerator and denominator in the fraction appearing in \eqref{eq:pf:lem:app:integrability estimate:fraction}. For the numerator, we use that by the growth condition \eqref{eq:growth.g} of $g$, there are $\tilde{x}$ and $C > 0$ such that
\begin{equation*}
G(x)\geqslant C x,\quad \forall x\geqslant \tilde{x}.
\end{equation*}
Using this, we obtain
\begin{align}
\int_0^\infty y^{q^\prime} \exp\left(-2 \frac{L_s}{M_s} G(y)\right)dy &\leqslant \frac{\tilde x^{q^\prime +1}}{1+ q^\prime} + \int_{\tilde x}^\infty y^{q^\prime} \exp\left(-2 \frac{L_s}{M_s} C y \right)dy \notag \\
& \leqslant \frac{\tilde x^{q^\prime +1}}{1+ q^\prime} + \int_{0}^\infty y^{q^\prime} \exp\left(-2 \frac{L_s}{M_s} C y \right)dy \notag \\
& \leqslant \frac{\tilde x^{q^\prime +1}}{1+ q^\prime} + \left(2 C \frac{L_s}{M_s}\right)^{-q^\prime -1} \int_{0}^\infty z^{q^\prime} \exp\left(-z \right)dy \notag \\
& \leqslant \frac{\tilde x^{q^\prime +1}}{1+ q^\prime} + \Gamma(q^\prime+1) (2 C)^{-q^\prime -1}  \left(\frac{M_s}{L_s}\right)^{q^\prime +1} \notag \\
&\leqslant \widetilde C \left(1 +  \left(\frac{M_s}{L_s}\right)^{q^\prime +1}\right),
\label{eq:pf:lem:app:integrability estimate:numerator}
\end{align}
for some constant $\widetilde C > 0$. For the denominator in the fraction appearing in \eqref{eq:pf:lem:app:integrability estimate:fraction}, we use that $G$ is continuous and nondecreasing on $\bR_+$ with $G(0) = 0$ and $\lim_{x\to\infty}G(x)=\infty$. Thus, there is $c > 0$ such that $G(c) \leqslant \frac{\eta}{4 p (1 + \eta)}$. This gives
\begin{align}
\int_{0}^\infty \exp\left(-2 \frac{L_s}{M_s} G(y)\right)dy &\geqslant \int_{0}^{c}\exp\left(-2 \frac{L_s}{M_s} G(y)\right)dy \geqslant \int_{0}^{c}\exp\left(-2 \frac{L_s}{M_s} G(c)\right)dy \notag \\
&= c\exp\left(-\frac{L_s}{M_s} G(c) \right) \geqslant c\exp\left(-\frac{L_s}{M_s} \frac{\eta}{4 p (1 + \eta)} \right).
\label{eq:pf:lem:app:integrability estimate:denominator}
\end{align}
It is an elementary exercise in analysis to show that there is a constant $A > 0$ such that 
\begin{equation*}
\left(1 + x^{-q^\prime - 1}\right) \exp\left(\frac{\eta}{4 p (1 + \eta)}x \right) \leqslant A\left(x^{-q^\prime - 1} + \exp\left(\frac{\eta}{4 p (1 + \eta)}x \right)\right), \quad  x > 0.
\end{equation*}
Together with \eqref{eq:pf:lem:app:integrability estimate:numerator} and \eqref{eq:pf:lem:app:integrability estimate:denominator}, it follows that there is a constant $\bar C \geqslant 1$ such that 
\begin{equation}
\label{eq:pf:lem:app:integrability estimate:fraction 2}
\frac{\int_0^\infty y^{q^\prime} \exp\left(-2 \frac{L_s}{M_s} G(y)\right)dy}{\int_0^\infty \exp\left(-2 \frac{L_s}{M_s} G(y)\right)dy} \leqslant \bar C \left(\left(\frac{M_s}{L_s}\right)^{q^\prime +1} +\exp\left(\frac{L_s}{M_s} \frac{\eta}{4p (1 + \eta)}\right)\right).
\end{equation}
Now putting together \eqref{eq:pf:lem:app:integrability estimate:fraction} and \eqref{eq:pf:lem:app:integrability estimate:fraction 2}, and using the elementary inequality $a b c \leqslant a^2 + b^4 + c^4$ for $a, b, c > 0$, we obtain
\begin{align}
&H_s \frac{\int_{\bR} f\left(\frac{K_s}{M_s}y\right) \exp\left(-2 \frac{L_s}{M_s} G(y)\right)dy}{\int_{\bR}\exp\left(2 \frac{L_s}{M_s} G(y)\right)dy} \notag \\
&\qquad \leqslant C_f \bar C\left(H_s  + H_s K_s^{q^\prime} \left(\frac{1}{L_s}\right)^{q^\prime} \frac{M_s}{L_s}  + H_s K_s^{q^\prime} \left(\frac{1}{M_s}\right)^{q^\prime}\exp\left(\frac{L_s}{M_s} \frac{\eta}{4 p (1 + \eta)}\right)\right) \notag \\
&\qquad \leqslant C_f \bar C\left(H_s  + H_s K_s^{q^\prime} \left(\frac{1}{L_s \wedge M_s}\right)^{q^\prime} \frac{M_s}{L_s}  + H_s K_s^{q^\prime} \left(\frac{1}{L_s \wedge M_s}\right)^{q^\prime}\exp\left(\frac{L_s}{M_s} \frac{\eta}{4 p (1 + \eta)}\right)\right) \notag \\
&\qquad \leqslant C_f \bar C \left( H_s + 2 \left(H_s K_s^{q^\prime}\right)^2 +  2\left(\frac{1}{L_s \wedge M_s}\right)^{4 q^\prime} + \left(\frac{M_s}{L_s}  \right)^4 + \exp\left(\frac{L_s}{M_s} \frac{\eta}{p (1 + \eta)}\right)\right) 
\label{eq:pf:lem:app:integrability estimate:as}
\end{align}
The claimed estimate~\eqref{eq:lem:app:integrability estimate} in turn follows from \eqref{eq:pf:lem:app:integrability estimate:as}, Assumption \ref{ass.integrability} and the elementary inequality $(a + b + c+d)^{p (1 + \eta)} \leqslant 4^{p (1 + \eta)}  \left(a^{p (1 + \eta)}  + b^{p (1 + \eta)} + c^{p (1 + \eta)}+d^{p (1 + \eta)}\right)$ for $a, b, c, d \geqslant 0$.
\end{proof}
\bibliographystyle{abbrv}
\bibliography{chm}
\end{document}